\documentclass[EJP,preprint]{ejpecp} 

\usepackage[capitalize,sort]{cleveref}

\SHORTTITLE{Tight Concentration Inequality for Sub-Weibull RVs with GBO norms}

\TITLE{Tight Concentration Inequality for Sub-Weibull Random Variables with Generalized Bernstien Orlicz norms
} 

\AUTHORS{%
  Heejong~Bong
    \footnote{Department of Statistics \& Data Science, Carnegie Mellon University, Pittsburgh, Pennsylvania, United States of America.
    \EMAIL{{hbong, arunku}@stat.cmu.edu}}
  \and 
  Arun~Kumar~Kuchibhotla
    \footnotemark[1]
}

\KEYWORDS{concentration inequality; sub-Weibull random variables; generalized Bernstein Orlicz norm; heavy-tailed distribution; Paley-Zygmund inequality} 

\AMSSUBJ{60G50; 60E15} 
\AMSSUBJSECONDARY{60B20; 62E22} 

\SUBMITTED{February 25, 2023} 

\ARXIVID{2302.03850} 

\VOLUME{0}
\YEAR{2020}
\PAPERNUM{0}
\DOI{10.1214/YY-TN}

\ABSTRACT{Recent development in high-dimensional statistical inference has necessitated concentration inequalities for a broader range of random variables. We focus on sub-Weibull random variables, which extend sub-Gaussian or sub-exponential random variables to allow heavy-tailed distributions. This paper presents concentration inequalities for independent sub-Weibull random variables with finite Generalized Bernstein-Orlicz norms, providing generalized Bernstein's inequalities and Rosenthal-type moment bounds. The tightness of the proposed bounds is shown through lower bounds of the concentration inequalities obtained via the Paley-Zygmund inequality. The results are applied to a graphical model inference problem, improving previous sample complexity bounds.}

\newcommand{\reals}{\ensuremath{\mathbb{R}}}

\newcommand{\textand}{\enskip \text{and} \enskip}
\def\hat#1{\widehat{#1}}

\def\bar#1{\overline{#1}}

\newcommand{\tsum}{{\textstyle\sum}}

\let\abs\undefined 
\DeclarePairedDelimiter\abs{\lvert}{\rvert}

\let\norm\undefined 
\DeclarePairedDelimiter\norm{\lVert}{\rVert}
\DeclarePairedDelimiter\mnorm{\lvert\kern-0.25ex\lvert\kern-0.25ex\lvert}{\rvert\kern-0.25ex\rvert\kern-0.25ex\rvert}

\DeclarePairedDelimiter\ceil{\lceil}{\rceil}
\DeclarePairedDelimiter\floor{\lfloor}{\rfloor}

\renewcommand{\Pr}{\mathbb{P}}  
\newcommand{\Exp}{\mathbb{E}}	

\begin{document}



\section{Introduction}

Concentration inequalities provide essential tools for theoretical analyses of statistical inference.
Since Sergei Bernstein introduced the Cramér-Cernoff technique in 1927, remarkable development has been made for concentration inequalities. 
Some of the most popular establishments include Bernstein's inequality, Hoeffding's inequality \cite{hoeffding1963probability}, and Bennet's inequality \cite{bennett1962probability}. Those inequalities result from the Cramér-Cernoff technique applied to the moment-generating functions (MGFs) of the random variables.
Those inequalities result in exponential tail bounds for sums of independent random variables, which form a foundation of theoretical analyses in high-dimensional statistics.
The classical forms of the inequalities were established on bounded random variables, but later, their assumptions were relaxed to the sub-Gaussian and sub-exponential conditions (see Sections 2.6-2.8 in \cite{vershynin2018} and Section 2.1 in \cite{wainwright2019high} for thorough reviews.) More recent investigations verified the optimality of the inequalities. For example, both Hoeffding's and Bennett's inequalities cannot be improved for sums of Bernoulli random variables (Lemma 4.7 in \cite{bentkus2004hoeffding}; Example 2.4 in \cite{major2005tail}). Recently, Zhang and Zhou \cite{zhang2020non} showed the general tightness of the MGF-based Cramér-Chernoff bounds. 

The recent development in high-dimensional statistical inference has necessitated concentration inequalities for a broader range of random variables, including those which do not have finite MGF in any neighborhood of $t = 0$. For example, in the context of inference in graphical models (e.g., Chapter 2 of \cite{bong2022discovery} and \cref{sec:application} of this paper), the leading error term is the bias due to the uncertainty in covariance matrix estimation. The uncertainty is quantified as the fourth power of the observed data. 
If the data are Gaussian, the MGF of the fourth power is not finite at $t \neq 0$,
so the MGF-based concentration inequalities are not applicable. In this sense, we say the bias term is heavy-tailed. To bind heavy-tailed random variables, we need concentration inequalities for generalized tail probability characteristics. Here we focus on sub-Weibull random variables, which have exponential tails like sub-Gaussian or sub-Exponential random variables.
A random variable $X$ is said to be {\it sub-Weibull of order $\alpha > 0$} or {\it sub-Weibull$(\alpha)$} with parameters $\nu$ and $L$ if
\begin{equation} \label{eq:sub_weibull_tail}
    \Pr[\abs{X} \geq t] \leq \begin{cases}
        2 \exp\left(-t^2/\nu^2\right), 
        & 0 \leq t < \nu L^{-\alpha/(2-\alpha)}, \\
        2 \exp\left(-t^\alpha/(L\nu)^\alpha\right),
        & t\ge \nu L^{-\alpha/(2-\alpha)}.
    \end{cases}
\end{equation}
A sub-Weibull random variable $X$ has a sub-Gaussian tail for small $t$, where the order $\alpha$ determines the size of tail at large $t$.
For $\alpha = 2$ and $1$, sub-Weibull$(\alpha)$ generalizes sub-Gaussian and sub-exponential, respectively \cite{wellner2017bennett}. 
If $\alpha$ is smaller than $1$, sub-Weibull$(\alpha)$ random variables are heavy-tailed. In particular, the fourth power bias term in our example is sub-Weibull$(1/2)$.

An equivalent way of configuring the sub-Weibull tail behaviors is by Orlicz norms. 
Given a non-decreasing function $g : [0, \infty) \rightarrow [0,\infty)$, the $g$-Orlicz norm of a random variable $X$ is defined by
\begin{equation*}
    \norm{X}_{g} := 
    \inf\{\eta > 0: \Exp[g(\abs{X}/\eta)] \leq 1\}. 
\end{equation*}
The $g$-Orlicz norm representation is equivalent to the tail bound 
\begin{equation*}
    \Pr[\abs{X} \geq t] \leq \frac{2}{g\left(t/\norm{X}_g\right)+1}, ~\forall t \geq 0.
\end{equation*}
Because \cref{eq:sub_weibull_tail} implies $\Pr[\abs{X} \geq t] \leq 2 \exp( - \min\{\frac{t^2}{\nu^2}, \frac{t^\alpha}{(\lambda \nu)^\alpha}\} )$, the tail bound leads to the Generalized Bernstein-Orlicz (GBO) norm of \cite{kuchibhotla2018moving}, the Orlicz norm with respect to $\phi_{\alpha,L}(x) := \exp(\min\{x^2, (x/L)^\alpha\}) - 1$. 
We note that the parameter $\nu$ corresponds to the size of $\norm{X}_{\phi_{\alpha,L}}$. We denote the collection of mean zero random variables with $\norm{X}_{\phi_{\alpha,L}} \leq \nu$ by $\Phi_{\alpha,L}(\nu)$.

In \cite{kuchibhotla2018moving}, Kuchibhotla and Chakrabortty provided a generalized version of Bernstein's inequality for sub-Weibull random variables and its application to high-dimensional statistical theory when sub-Weibull$(\alpha)$ was defined with $\psi_\alpha$-Orlicz norm, where $\psi_\alpha(x) = \exp(x^\alpha) - 1$. We note that random variable $X$ with finite $\psi_\alpha$-Orlicz norm has an unmixed tail bound, $\Pr[|X| \geq t] \leq 2\exp(-{t^\alpha}/{\norm{X}_{\psi_\alpha}^\alpha})$ so that $X$ is sub-Weibull$(\alpha)$ (i.e., satisfies \cref{eq:sub_weibull_tail}) with parameters $\nu = \norm{X}_{\psi_\alpha}/L$ and $L$ for any $L > 0$. To exemplify the advantages offered by the granularity of the GBO norm, consider the tail bound implied by Theorem 3.1 of \cite{kuchibhotla2018moving}: If $X_1, \dots, X_n$ are independent mean zero random variables with $\norm{X_i}_{\psi_\alpha} < \infty$, then for any vector $a = (a_1, \dots, a_n) \in \reals^n$,
\begin{equation*}
    \Pr\left[\abs*{\sum_{i=1}^n a_i X_i} \geq C(\alpha) (\norm{b}_2 \sqrt{t} + \norm{b}_{\beta(\alpha)} t^{1/\alpha})\right] \leq 2e^{-t}, 
\end{equation*}
where $b = (a_1 \norm{X_1}_{\psi_\alpha}, \dots, a_n \norm{X_n}_{\psi_\alpha})$, $\beta(\alpha) = \infty$ when $\alpha \leq 1$, $\beta(\alpha) = \alpha/(\alpha-1)$ when $\alpha > 1$, and $C(\alpha)$ is a constant dependent on $\alpha$.
The tail bound is a mixture of sub-Gaussian (when $t$ is small) and sub-Weibull$(\alpha)$ (when $t$ is large) tails and is better represented in \cite{kuchibhotla2018moving} by the GBO norm, rather than the $\psi_\alpha$ norm. If $\psi_\alpha$-Orlicz norm were used, the tightness of the tail at large $t$ should be compromised by upper bounding the sub-Gaussian tail at small $t$. This is due to a better granularity of the sub-Weibull space represented by the GBO norm than by $\psi_\alpha$-Orlicz norm. To have a finite $\psi_\alpha$-Orlicz norm, random variable $X$ should satisfy \cref{eq:sub_weibull_tail} with every value of $L > 0$, with $\nu = \|X\|_{\psi_{\alpha}}/L$. In contrary, the GBO norm has one more degree of freedom with parameter to further granulate the space of sub-Weibull$(\alpha)$ random variables. 

In this paper, we extend the results of \cite{kuchibhotla2018moving} to sub-Weibull$(\alpha,L)$ random variables with finite Generalized Bernstein-Orlicz (GBO) norms. The added granularity of the specification for the summands $X_1, \dots, X_n$ benefits the tightness of theoretical analyses in many statistical applications; we discuss an example in \cref{sec:application}.
We consider concentration inequalities in moments, the GBO norms, and tail probability bounds for weighted sums of independent $\text{sub-Weibull}(\alpha,L)$ random variables. To bound the moments, we follow the {\it sequence Orlicz norm} argument devised by Lata{\l}a in \cite{latala1997estimation}, which we provide as preliminary theoretical results in \cref{sec:pre_lemmas}. The resulting moment bound follows the form of Rosenthal's inequality (see \cref{thm:moment_bound}). From the moment bounds, we induce the upper bounds for the GBO norms and tail probabilities, using the arguments in \cite{kuchibhotla2018moving}. The resulting tail probability bounds provide generalized Bernstein's inequalities for $\text{sub-Weibull}(\alpha,L)$ random variables. We also show the tightness of the proposed bounds by obtaining the lower bounds of the concentration inequalities in the worst case. The lower bound proof follows the Paley-Zygmund inequality technique in the corollary of \cite{gluskin1995tail} and Theorem 2.1 of \cite{hitczenko1999note}.

Our theoretical contribution focuses on the case of $\alpha \leq 1$. Tail probabilities of $\text{sub-Weibull}(\alpha,L)$ random variables with order $\alpha > 1$ are bounded by a mixture of logarithmically concave (log-concave) tails. 
For this case, the tight moment and tail probability bounds are readily available by \cite{gluskin1995tail}, 
which we also include in this paper for completeness. We refer to Theorem 1 of \cite{hitczenko1997moment} for log-convex tail cases and Examples 3.2 and 3.3 in \cite{latala1997estimation} for the sequence Orlicz norm technique establishing the tight moment bounds in the two references.

\section{Main Results} \label{sec:main_results}
Let $X_i$ be independent mean zero sub-Weibull$(\alpha,L_i)$ random variables for $i=1, \dots, n$.
Our objective is to bound the size of a weighted sum $X^* := \sum_{i=1}^n a_i X_i$ for $a = (a_1, \dots, a_n) \in \reals^n$. 
Because $\norm{X_i/c}_{\phi_{\alpha,L_i}} = \norm{X_i}_{\phi_{\alpha,L_i}}/c$ for $c > 0$, by adjusting the value of $a_i$,
we assume $\norm{X_i}_{\phi_{\alpha,L_i}} \leq 1$ and $a_i \geq 0$. We denote the collection of such random variables by $\Phi_{\alpha,L_i}(1)$. More generally, for $\eta > 0$,
\begin{equation*}
    \Phi_{\alpha, L}(\eta) := \{ X: \Exp[X] = 0, \norm{X}_{\phi_{\alpha,L}} \leq \eta \}.
\end{equation*}
For $L = (L_1, \dots, L_n)$, we denote the collection $(X_1, \dots, X_n)$ of random variables satisfying the above conditions by 
\begin{equation}
    \label{eq:collection-sub-Weibull-notation}
    \mathcal{X}_{\alpha, L} = \{(X_1, \dots, X_n): \Exp[X_i] = 0, ~ X_i \in \Phi_{\alpha, L_i}(1), ~\text{and}~ X_1, \ldots, X_n\text{ are independent}\}.
\end{equation}
Also, because the order of the random variables does not change the sum, we will assume $(a_i \bar{L}_i)_{i=1}^n$ is non-decreasing, where $\bar{L} = (\max\{1, L_i\})_{i=1}^{n}$. That is, $a_1 \bar{L}_1 \geq \dots \geq a_n \bar{L}_n \geq 0$.

Among the collection $\mathcal{X}_{\alpha,L}$, there are random variables that take a special place; these are random variables that attain the worst case upper bounds. Let $Z_1, Z_2, \ldots, Z_n$ be independent random variables satisfying 
\begin{equation} \label{eq:Z_sub_weibull_tail}
    \Pr[\abs{Z_i} \geq t] = \exp( - \min\{t^2, (t/L_i)^\alpha\} ), ~ \forall t \geq 0, ~ i=1, \dots, n,
\end{equation}
and $Z^* := \sum_{i=1}^n a_i Z_i$. We note that $\min\{\sqrt{2}, 2^{1/\alpha}\} \leq \norm{Z_i}_{\phi_{\alpha, L_i}} \leq \max\{\sqrt{3}, 3^{1/\alpha}\}$; see \cref{sec:GBO_Z_i} for the proof. Also, the proofs for the main results are provided in \cref{sec:pf_theorems}.

\vspace{0.1in}
{\bf Notation.} 
The binary operator $\odot$ between two vectors indicates the element-wise multiplication, i.e., for two vectors $a, b \in \reals^n$, $a \odot b = (a_i b_i)_{i=1}^n$. 
In the following results, $C(\dots)$ is a constant with implicit dependency to the parameters in the parentheses, whose value changes across lines. We omit the parentheses for absolute constants with no dependency on problem parameters and denote them by $C$.

\subsection{Optimal Moment Bound} \label{sec:moment_bound}

First, we provide a tight upper bound for the moments of $X^*$. Following the proof of Theorem 3.1 in \cite{kuchibhotla2018moving}, we show that $\norm{X^*}_p$ is upper bounded by $C(\alpha) \norm{Z^*}_p$. On the other hand, $Z_i/\max\{\sqrt{3}, 3^{1/\alpha}\} \in \Phi_{\alpha,L_i}(1)$, so $\norm{Z^*/\max\{\sqrt{3}, 3^{1/\alpha}\}}_p$ lower bounds $\sup_{X_i \in \Phi_{\alpha,L_i}(1)} \norm{X^*}_p$. Therefore, the tight moment bounds for $Z^*$ are sufficient for deriving those for $X^*$. For $\alpha \leq 1$, the tail probability $\Pr[\abs{Z_i} \geq t]$ is log-concave for small $t$ and logarithmically convex (log-convex) for large $t$, which are addressed in Examples 3.2 and 3.3 of \cite{latala1997estimation}, respectively. For $\alpha \geq 1$, the tail probability is bounded by two log-concave functions. The tight moment concentration inequality for log-concave-tailed random variables is available in \cite{gluskin1995tail}. We use a slightly generalized result in Example 3.2 of \cite{latala1997estimation}, which didn't require i.i.d. assumption made in \cite{gluskin1995tail}. The resulting moment bound is given as follows.

\begin{theorem} \label{thm:moment_bound}
    Fix $\alpha\ge 0$ and non-negative real numbers $L_1, \ldots, L_n$.
    Then for $p \geq 1$, 
    \begin{equation} \label{eq:moment_upper}
        \sup_{(X_1,\dots,X_n) \in \mathcal{X}_{\alpha,L}} \norm{X^*}_p
        ~\leq~ C(\alpha) \max\left\{\sqrt{p} \norm{a \odot \bar{L}}_2,\;
        p^{1/\alpha} \norm{(a_i L_i: i\leq p)}_\beta \right\},
    \end{equation}
    where $\beta = \frac{\alpha}{\alpha-1}$ for $\alpha > 1$ and $\infty$ for $\alpha \leq 1$.
    Furthermore, the upper bound is tight in terms of that for $p \geq 1$,
    \begin{equation} \label{eq:moment_lower}
        \sup_{(X_1,\dots,X_n) \in \mathcal{X}_{\alpha,L}} \norm{X^*}_p
        ~\geq~ \frac{1}{C(\alpha)} \max\left\{\sqrt{p} \norm{a \odot \bar{L}}_2,\;
        p^{1/\alpha} \norm{(a_i L_i: i\leq p)}_\beta \right\}.
    \end{equation}
\end{theorem}
The bound is in the form of the Rosenthal inequalities. We focus on the dependency of the constants on order $p$ of the moments, while the dependency on $\alpha$ is given implicitly. The appearance of $\bar{L}$ is anticipated because sub-Weibull$(\alpha,L)$ random variables have a mixture of sub-Gaussian and sub-Weibull$(\alpha)$ tail probability bounds.
The tail behavior of $X^*$ follows both sub-Gaussian and sub-Weibull$(\alpha)$ concentration inequalities. Based on \cite{kuchibhotla2018moving}, the sub-Gaussian concentration inequality induces a $p$-moment upper bound $C(\alpha) \sqrt{p} \norm{a}_2$, whereas the sub-Weibull$(\alpha)$ version yields $C(\alpha) \max\{ \sqrt{p} \norm{a \odot L}_2, p^{1/\alpha} \norm{(a_i L_i: i \leq p)}_\beta \}$. Taking their maximum, we have $\bar{L}$ term inside the $\norm{\cdot}_2$-norm. This term is confirmed to be essential by the lower bound in \cref{eq:moment_lower},
which demonstrates the tightness of the upper bound. We note that the supremum is attained (up to a constant) when $X_i = Z_i/\max\{\sqrt{3}, 3^{1/\alpha}\}$. 

\begin{remark}
Suppose that $X_i$ are sub-Weibull$(\alpha)$ random variables with $\norm{X_i}_{\psi_\alpha} \leq 1$ and that $\abs{a_1} \geq \dots \geq \abs{a_n}$. By the definitions of the $\psi_\alpha$- and $\phi_{\alpha,L}$-Orlicz norms, 
\begin{equation*}
    \norm{X_i}_{\psi_\alpha} \leq 1 \Rightarrow \norm{L_i X_i}_{\phi_{\alpha,L_i}} \leq 1,
\end{equation*}
for any $L_i > 0$. Applying \cref{eq:moment_upper} to $X^* = \sum_{i=1}^n a_i X_i = \sum_{i=1}^n \frac{a_i}{L_i} (L_i X_i)$, we obtain 
\begin{equation} \label{eq:incomplete_moment_upper_psi}
    \norm{X^*}_p \leq C(\alpha) \max\left\{ \sqrt{p} \norm{(a_i \max\{1, 1/L_i\}: i \in [n])}, p^{1/\alpha} \norm{(a_i: i \leq p)}_\beta \right\}. 
\end{equation}
Because \cref{eq:incomplete_moment_upper_psi} holds for every $L_i > 0$, we have a concentration inequality for sub-Weibull$(\alpha)$ random variables with finite $\psi_\alpha$ norms.
\begin{equation} \label{eq:moment_upper_psi}
    \norm{X^*}_p \leq C(\alpha) \max\left\{ \sqrt{p} \norm{a}_2, p^{1/\alpha} \norm{(a_i: i \leq p)}_\beta \right\}. 
\end{equation}
This upper bound is better than those in the proof of Theorem 3.1 in \cite{kuchibhotla2018moving}. 
In \cite{kuchibhotla2018moving}, the moment bound was written only in terms of $\|a\|_{\beta}$ and did not involve $\norm{(a_i: i\leq p)}_\beta$.
Because
\[
\|(a_i:\,1\le i\le p)\|_{\beta} = \begin{cases}
\|(a_i:\, 1\le i\le p)\|_{\infty} = \|a\|_{\infty} = \|a\|_{\beta}, &\mbox{if }\alpha \le 1,\\
\|(a_i:\, 1\le i\le p)\|_{\beta} \le \|a\|_{\beta}, &\mbox{if }\alpha > 1,
\end{cases}
\]
we get that the moment upper bound in \cite{kuchibhotla2018moving} was suboptimal for $\alpha > 1$.
\end{remark}

\subsection{Optimal GBO Norm Bound} \label{sec:GBO_bound}

Next, we induce GBO norm bounds for $X^*$ from the moment bound results in \cref{sec:moment_bound}. We use the equivalence between the moment growth and the GBO norm, provided in Proposition A.4 of \cite{kuchibhotla2018moving}.
The resulting GBO norm bound is given as follows.

\begin{theorem} \label{thm:GBO_bound}
    There exists a constant $C(\alpha)$ depending only on $\alpha$ such that with $L^* := C(\alpha) \norm{a \odot L}_\beta/\norm{a \odot \bar{L}}_2$, 
    \begin{equation} \label{eq:GBO_upper}
        \sup_{(X_1,\dots,X_n) \in \mathcal{X}_{\alpha,L}} \norm{X^*}_{\phi_{\alpha, L^*}} 
        ~\leq~ C(\alpha) \norm{a \odot \bar{L}}_2,
    \end{equation}
    where $\beta = \frac{\alpha}{\alpha-1}$ for $\alpha > 1$ and $\infty$ for $\alpha \leq 1$.
    Furthermore, the upper bound is tight in terms of that
    \begin{equation} \label{eq:GBO_lower}
        \sup_{(X_1,\dots,X_n) \in \mathcal{X}_{\alpha,L}}  \norm{X^*}_{\phi_{\alpha, L^*}} 
        ~\geq~ \frac{1}{C(\alpha)} \norm{a \odot \bar{L}}_2.
    \end{equation}
\end{theorem}

The supremum in~\cref{eq:GBO_lower} is attained (up to constants) when $X_i = Z_i/\max\{\sqrt{3}, 3^{1/\alpha}\} \in \Phi_{\alpha,L_i}(1)$.

\subsection{Optimal Tail Probability Bound} \label{sec:tail_bound}

Last, we induce the tail probability bounds for $X^*$ from the moment bound in \cref{sec:moment_bound}. The upper bound follows from Chebyshev's inequality, and the lower bound  from the Paley-Zygmund inequality. The proof follows Corollary of \cite{gluskin1995tail} and Theorem 2.1 of \cite{hitczenko1999note}.

\begin{theorem} \label{thm:tail_bound}
    Let $K(t) = \sup\{p: \sqrt{p} \norm{a \odot \bar{L}}_2 + p^{1/\alpha} \norm{(a_i L_i: i\leq p)}_\beta \leq t\}$, where $\beta = \frac{\alpha}{\alpha-1}$ for $\alpha > 1$ and $\infty$ for $\alpha \leq 1$.
    Then, for all $\alpha > 0$, there exists a constant $C(\alpha)$ such that for all $t \geq \norm{a \odot \bar{L}}_2 + \norm{a \odot L}_\infty$, 
    \begin{equation} \label{eq:tail_prob_upper_logconcave}
        \sup_{(X_1,\dots,X_n) \in \mathcal{X}_{\alpha,L}} \Pr\left[\abs*{X^*} \geq t\right] 
        ~\leq~ \exp\left(- \frac{1}{C(\alpha)} K(t)\right).
    \end{equation}
    Furthermore, the upper bound is tight in terms of that for $t \geq \norm{a \odot \bar{L}}_2 + \norm{a \odot L}_\infty$,
    \begin{equation} \label{eq:tail_prob_lower_logconcave}
        \sup_{(X_1,\dots,X_n) \in \mathcal{X}_{\alpha,L}} \Pr\left[\abs*{X^*} \geq t\right] 
        ~\geq~ \exp\left(- C(\alpha) K(t)\right).
    \end{equation}
    Alternatively, for $t \geq 0$,
    \begin{equation} \label{eq:tail_prob_upper}
    \begin{split}
        &\sup_{(X_1,\dots,X_n) \in \mathcal{X}_{\alpha,L}} \Pr\left[\abs*{X^*} \geq t\right] 
        ~\leq~ 2 \exp\left(- \frac{1}{C(\alpha)} \min\left\{
        \frac{t^2}{\norm{a \odot \bar{L}}_2^2}, 
        \frac{t^{\alpha}}{\norm{a \odot L}_\beta^\alpha}\right\}\right),
    \end{split}
    \end{equation}
    and the upper bound is tight if $\alpha \leq 1$: for $t \geq 0$,
    \begin{equation} \label{eq:tail_prob_lower}
    \begin{split}
        &\sup_{(X_1,\dots,X_n) \in \mathcal{X}_{\alpha,L}} \Pr\left[\abs*{X^*} \geq t\right] 
        ~\geq~ \frac{1}{C(\alpha)} \exp\left(- C(\alpha) \min\left\{
        \frac{t^2}{\norm{a \odot \bar{L}}_2^2}, 
        \frac{t^{\alpha}}{\norm{a \odot L}_\beta^\alpha}\right\}\right).
    \end{split}
    \end{equation}
\end{theorem}

Inequalities~\cref{eq:tail_prob_upper_logconcave} and~\cref{eq:tail_prob_lower_logconcave} provide an optimal tail bounds for all $\alpha > 0$ and all $t\ge\|a\odot\overline{L}\|_2 + \|a\odot L\|_{\infty}$. Unfortunately, these tail bounds involve $K(t)$ which is not available in closed form because the moment bound includes $(a_iL_i:1\le i\le p)$. Inequality~\cref{eq:tail_prob_upper} provides a simple tail upper bound for all $\alpha > 0$. While this upper bound may not be optimal for all $\alpha > 0$, it is a tight bound when $\alpha \le 1$. This tightness comes from the fact that $\norm{a \odot L}_\infty = \norm{(a_i L_i: i\leq p)}_\beta = \norm{a \odot L}_\beta$ for $\alpha \leq 1$.

\section{Application} \label{sec:application}
In this section, we apply the techniques and results developed in previous sections to derive dimension/sample complexity in an application related to covariance matrix estimation. 
Suppose we observe $m$ sets of $p$-dimensional random vectors and estimate the covariance matrix within each sample set. Namely, for $l \in [m]$, $X^{(l)}_1, \dots, X^{(l)}_{n}$ are independent random vectors with mean zero and covariance matrix $\Sigma^{(l)}$. A straightforward estimator of $\Sigma^{(l)}$ is the empirical covariance matrix $\hat{\Sigma}^{(l)} = \frac{1}{n} \sum_{k=1}^{n} X^{(l)}_k X^{(l)\top}_k$.
We attempt to identify the shared covariance structure among $\{\Sigma^{(l)}: l \in [m]\}$ based on $\{\hat\Sigma^{(l)}: l \in [m]\}$. An important challenge in statistical inference is estimating heterogeneity across the $m$ estimators. 
For example, in Chapter 2 of \cite{bong2022discovery}, the author developed and studied a method of identifying the common graphical structure across multiple sets of matrix-variate random samples. Their theoretical analyses configured the leading term of the inference error to be $\sup_{i,j \in [p]} \sum_{l=1}^m (\hat\Sigma_{ij}^{(l)} - \Sigma_{ij}^{(l)})^2$, which evaluates the variability of the estimation error in $\hat{\Sigma}^{(l)}$ across $l \in [m]$. It is desired to configure the tight concentration inequality for the error term.

For simplicity of the analysis, we assume each entry of the random vectors is sub-Gaussian: $\norm{X^{(l)}_{k,i}}_{\psi_{2}} \leq 1$ for $i \in [p]$, $k \in [n]$ and $l \in [m]$. Recall that $\norm{\cdot}_{\psi_{\alpha}}$ is the Orlicz norm with $\psi_{\alpha}(x) = \exp(x^\alpha) - 1$.
Based on Proposition D.2 of \cite{kuchibhotla2018moving} and the centering lemma (Exercise 2.7.10, \cite{vershynin2018}), $X_{k,i}^{(l)} X_{k,j}^{(l)}$ is sub-exponential: $\norm{X_{k,i}^{(l)} X_{k,j}^{(l)} - \Sigma_{ij}^{(l)}}_{\psi_1} \leq C$ for some universal constant $C > 0$, and Theorem 3.1 of \cite{kuchibhotla2018moving} implies that
\begin{equation} \label{eq:GBO_sigma_error}
    \norm*{\hat\Sigma_{ij}^{(l)} - \Sigma_{ij}^{(l)}}_{\phi_{1, \frac{C}{\sqrt{n}}}}
    = \norm*{\frac{1}{n} \sum_{k=1}^{n} X_{k,i}^{(l)} X_{k,j}^{(l)} - \Sigma_{ij}^{(l)}}_{\phi_{1, \frac{C}{\sqrt{n}}}}
    \leq \frac{C}{\sqrt{n}}.
\end{equation}
It is easy to check that $(\hat\Sigma_{ij}^{(l)} - \Sigma_{ij}^{(l)})^2$ has a mixture sub-exponential and sub-Weibull$(1/2)$ tails: namely,
\begin{equation} \label{eq:mixture_tail_sigma_se}
    \Pr[(\hat\Sigma_{ij}^{(l)} - \Sigma_{ij}^{(l)})^2 \geq t]
    \leq 2 \exp\left( - C \min\{ n t, n \sqrt{t} \}\right).
\end{equation}
Based on this observation, we obtain \cref{thm:tail_bound_sigma_sse} for the tail probability bound of $\sum_{l=1}^m (\hat\Sigma_{ij}^{(l)} - \Sigma_{ij}^{(l)})^2$, using the main results in \cref{sec:main_results}. See \cref{sec:pf_tail_bound_sigma_sse} for the proof details. It is important to note that the results in~\cref{sec:main_results} are not readily applicable (particularly, the lower bounds) because the tail in~\cref{eq:mixture_tail_sigma_se} is a mixture of exponential and sub-Weibull$(1/2)$ and not a mixture of Gaussian and sub-Weibull($1/2$) as required by the GBO norm.

\begin{theorem} \label{thm:tail_bound_sigma_sse}
For $\hat\Sigma$ satisfying the aforementioned conditions,
\begin{equation} \label{eq:upper_bound_sigma_sse}
\begin{aligned} 
    & \Pr\left[ \sum_{l=1}^m (\hat\Sigma_{ij}^{(l)} - \Sigma_{ij}^{(l)})^2 \geq t + C \frac{m}{n} \right]
    \leq 4 \exp\left(- C \min\left\{\frac{n^2t^2}{m}, nt, n\sqrt{t} \right\} \right), ~\forall t > 0,
\end{aligned}
\end{equation}
or equivalently,
\begin{equation*}
    \sum_{l=1}^m (\hat\Sigma_{ij}^{(l)} - \Sigma_{ij}^{(l)})^2 \leq C \left( \frac{m + \sqrt{m\nu} + \nu}{n} + \frac{\nu^2}{n^2} \right),
\end{equation*} 
with probability at least $1 - e^{-\nu}$ for any $\nu > 0$.
On the other hand, suppose that there exists $C > 0$ such that $\Pr[\abs{\hat\Sigma_{ij}^{(l)} - \Sigma_{ij}^{(l)}} \geq t] {\geq} \exp( - C \min\{ nt, nt^2\} ), ~ \forall t > 0$. Then,
\begin{equation} \label{eq:lower_bound_sigma_sse}
\begin{aligned}
    \Pr\left[\sum_{l=1}^m (\hat\Sigma_{ij}^{(l)} - \Sigma_{ij}^{(l)})^2 \geq t + C\frac{m}{n} \right]
    & \geq \frac{1}{C} \exp\left(- C \min\left\{\frac{n^2t^2}{m}, nt, n\sqrt{t} \right\} \right), ~\forall t > 0.
\end{aligned}    
\end{equation}
\end{theorem}

The required sample complexity for the convergence of the leading error term to $0$ in probability is ${m+\log(q)} = o({n})$, which is better than ${m+\log(q)^2} = o({n})$ in Lemma B.2.5 of \cite{bong2022discovery}. Moreover, \cref{eq:lower_bound_sigma_sse} shows the tightness of the upper bound. The condition $\Pr[\abs{\hat\Sigma_{ij}^{(l)} - \Sigma_{ij}^{(l)}} \geq t] {\geq} \exp( - C \min\{ nt, nt^2\} ), ~ \forall t > 0$ is feasible in covariance estimation. For example, suppose that $X_k^{(l)}$ is a multivariate Gaussian random vector with the identity covariance matrix. Then, for any $i \in [p]$, $\hat{\Sigma}_{ii}^{(l)}$ is a chi-squared distribution random variable with degree of freedom $n$. In this case, the lower bound for the tail probability lower bound satisfies the condition. We refer to Corollary 3 of \cite{zhang2020non} for the tail probability lower bound of chi-squared random variables.

\section{Discussion}

In this paper, we presented concentration inequalities in moments, the GBO norms, and tail probability bounds for weighted sums of independent $\text{sub-Weibull}(\alpha,L)$ random variables. The resulting moment bound was in form of Rosenthal's inequalities, and the tail probability bounds provided generalized Bernstein's inequalities for $\text{sub-Weibull}(\alpha,L)$ random variables. 
We also showed the tightness of the proposed bounds by obtaining the lower bounds of the concentration inequalities in the worst case.

We applied the main results to a graphical model inference problem. Although the summands of interest were not sub-Weibull and had a mixture of sub-exponential and sub-Weibull$({1}/{2}$) tails instead, we were still able to provide a tight bound for the inference error, using the main results. This improved the sample complexity in the original work (Chapter 2, \cite{bong2022discovery}) for the convergence of the leading error term to $0$ in probability. Likewise, the developed concentration inequalities will be useful in binding the size of summed random variables when the random variables have mixture tails, which U-statistics often entail with (e.g., see \cite{gine2000exponential} and \cite{bakhshizadeh2023exponential}).

There are several interesting future directions based on this work. Firstly, our bounds are derived for random variables under the condition of the GBO norm. Often in applications, it would be useful to take into account the variance of the random variables involved. In \cite{kuchibhotla2018moving}, the authors provided upper bounds for the tails and moments under the additional condition of variance. To our knowledge, lower bounds on the tail are not available for sums of independent sub-Weibull random variables with an additional condition on variance. Secondly, mixture of different tails appear naturally for $U$-statistics as shown in \cite{gine2000exponential}, but optimal tail bounds are not available unless the kernel is assumed to be bounded. It would be interesting to develop a parallel framework for $U$-statistics with general decay on the kernel.

\section{Proofs}

\subsection{Preliminary Lemmas} \label{sec:pre_lemmas}

Following the arguments in \cite{latala1997estimation}, we define
\begin{equation*}
\begin{aligned}
    \varphi_p(x) & := \frac{\abs{1+x}^p + \abs{1-x}^p}{2}, \\
    \phi_p(X) & := \Exp[\varphi_p(X)], \\
\end{aligned}
\end{equation*}
for $x \in \reals$ and a random variable $X$. We define the following Orlicz norm for a sequence of symmetric random variables $X_1, \dots, X_n$:
\begin{equation} \label{eq:sequence_Orlicz_norm}
    \mnorm{(X_i)}_p := \inf\{\eta > 0: \sum_i \log(\phi_p(X_i/\eta)) < p \}.
\end{equation}
Due to technical issues in the forthcoming proofs, we use a different definition of sequence Orlicz norm than that in \cite{latala1997estimation}, which was given by $\inf\{\eta > 0: \sum_i \log(\phi_p(X_i/\eta)) \leq p \}$. We note that $\phi_p(X/\eta)$ is strictly decreasing for $\eta \in (0, \infty)$ unless $X = 0$ almost surely, so the resulting value of $\mnorm{(X_i)}_p$ does not change. 
The following lemma references to \cite{latala1997estimation}. We note that the lemma holds for non-integer $p$.

\begin{lemma}[$p$-th moment bound; Theorem 2, \cite{latala1997estimation}]
\label{thm:moment_bound_by_Orlicz_norm}
    Let $X_1, \dots, X_n$ be a sequence of independent symmetric random variables, and $p \geq 2$. then,
    \begin{equation*}
        \frac{e-1}{2e^2} \mnorm{(X_i)}_p 
        \leq \norm{X_1 + \dots + X_n}_p
        \leq e \mnorm{(X_i)}_p.
    \end{equation*}
\end{lemma}

In \cite{latala1997estimation}, the authors demonstrated the use of \cref{thm:moment_bound_by_Orlicz_norm} in linear sums of iid symmetric random variables when the random variables have logarithmically concave or convex tails. (See the examples in Sections 3.2 and 3.3 therein.) Although $Z_i$ does not have such tails, we can use the results in \cite{latala1997estimation} based on the relationship between $Z_i$ and the following random variable. Let $Y_i$ be an iid symmetric sub-Gaussian random variable satisfying 
\begin{equation} \label{eq:Y}
    \Pr[\abs{Y_i} \geq t] = e^{-t^2}, ~~ \forall t \geq 0, ~~ i=1, \dots, n.
\end{equation}
Then, $Y_i$ relates to $Z_i$ through the following lemma:

\begin{lemma} \label{thm:Z_by_Y}
    \begin{equation}
        \abs{Z_i} \overset{d}{=} \max\{\abs{Y_i}, L_i \abs{Y_i}^{2/\alpha}\}.
    \end{equation}
\end{lemma}

See \cref{sec:pf_lemmas} for the proof.

\subsection{Proofs for Main Results} \label{sec:pf_theorems}

\begin{proof}[Proof of Theorem~\ref{thm:moment_bound}]

For the upper bound in \cref{eq:moment_upper}, we follow the argument of \cite{kuchibhotla2018moving} using the symmetrization inequality (Proposition 6.3 of \cite{ledoux1991probability}) and Theorem 1.3.1 of \cite{de2012decoupling}. If we set $\mu_i = \max\{\sqrt{\log 2}, L_i(\log 2)^{1/\alpha}\}$,
\begin{equation} \label{eq:prob_bound_by_Z}
\begin{aligned}
    & \Pr[\abs{X_i} \geq t] \leq 2 \exp(- \min\{t^2, (t/L_i)^\alpha\}) \\ 
    & \Longrightarrow \Pr[\max\{\abs{X_i} - \mu_i,0\} \geq t] \leq \exp(- \min\{t^2, (t/L_i)^\alpha\}).
\end{aligned}\end{equation}
Because $\mu_i \leq C(\alpha) \bar{L}_i$, for $p \geq 1$,
\begin{equation*}\begin{aligned}
    \norm{X^*}_p 
    & \leq 2 \norm*{\sum_{i=1}^n \varepsilon_i a_i \max\{\abs{X_i} - \mu_i,0\}}_p 
    + 2 \sqrt{p} \norm{a \odot \mu}_2 \\
    & \leq 2 \norm*{Z^*}_p 
    + C(\alpha) \sqrt{p} \norm{a \odot \bar{L}}_2. \\
\end{aligned}\end{equation*} 
For the lower bound in \cref{eq:moment_lower}, $Z_i/\max\{\sqrt{3}, 3^{1/\alpha}\} \in \Phi_{\alpha,L_i}(1)$, so $\norm{Z^*/\max\{\sqrt{3}, 3^{1/\alpha}\}}_p$ lowerbounds $\sup_{X_i \in \Phi_{\alpha,L_i}(1)} \norm{X^*}_p$. Therefore, it suffices to derive the tight moment bound for $Z^*$. We consider the cases of $\alpha > 1$ and $\alpha \leq 1$, separately.

\vspace{0.1in}
{\bf Cases of $\alpha > 1$:  } Let $Z'_i = Z_i/\bar{L}_i$ and $a'_i = a_i \bar{L}_i$ so that $\Pr[\abs{Z'_i} \leq 1] = e^{-1}$. We define $N_i(t) = - \log(\Pr[\abs{Z'_i} \geq t]) = \min\{N_{i,1}(t), N_{i,2}(t)\}$, where $N_{i,1}(t) = \bar{L}_i^2 t^2$, and $N_{i,2}(t) = (\frac{\bar{L}_i}{L_i} t)^\alpha$. We also define the dual functions $N_{i}^*(t)$, $N_{i,1}^*(t)$ and $N_{i,2}^*(t)$ of $N_i$, $N_{i,1}$ and $N_{i,2}$:
\begin{equation*}
    N_{i}^*(t) = \sup\{st - N_{i}(s): s > 0\} \textand
    N_{i,k}^{(t)} = \sup\{st - N_{i}(s): s > 0 \} \text{ for } k \in \{1,2\}.
\end{equation*}
In particular, $N_{i,1}^*(t) = \frac{t^2}{4\bar{L}_i^2}$ and $N_{i,2}^*(t) = (\alpha-1)(\frac{L_i t}{\bar{L}_i \alpha})^{\alpha/(\alpha-1)}$. Because they are suprema of linear (i.e., convex) functions, $N_i$, $N_{i,1}$ and $N_{i,2}$ are convex. We claim that $N_i^*(t) = \max\{N_{i,1}^*(t), N_{i,2}^*(t)\}$. The claim is justified by
\begin{equation*}
\begin{aligned}
    N_i^*(t) 
    & = \sup\left\{st - \min\left\{ N_{i,1}(t), N_{i,2}(t)
    \right\}\right\} 
    = \sup\left\{st - N_{i,k}(t): k=1,2
    \right\} \\
    & = \max\{\sup\{st - N_{i,1}(t)\}, \sup\{st - N_{i,2}(t)\}\} 
    = \max\{N_{i,1}^*(t), N_{i,2}^*(t)\}.
\end{aligned}
\end{equation*}

Although $N_i(t)$ is not convex, $N_i(st) \leq s N_i(t)$ for any $t > 0$ and $s \in [0,1]$. So Theorem 1 of \cite{gluskin1995tail} is applicable. (See Remark 4 therein.) Here we state a slightly generalized version that does not require i.i.d. assumption made in \cite{gluskin1995tail} and provide a standalone proof in \cref{sec:pf_lemmas}.

\begin{lemma}[A generalized version of Theorem 1, \cite{gluskin1995tail}]
\label{thm:gluskin}
Let $Z_1, \dots, Z_n$ be random variables with $\Pr[\abs{Z_i} \geq t] = e^{-N_i(t)}$ and $a_1, \dots, a_n$ be a nonincreasing sequence of nonnegative real numbers.
For each $i$, we define 
\begin{equation*}
    \kappa_{i} = \inf\left\{c > 0: \int_s^\infty e^{-N_i(t)} dt \leq \int_s^\infty e^{-t/c} dt, \text{ for all } s > 0 \right\}    
\end{equation*}
and assume that $N_i(1) = 1$ and $N_i(st) \leq s N_i(t), \forall t > 0, \forall s \in [0,1]$. For each $1 \leq p < \infty$,
\begin{equation*}
\begin{aligned}
    & C_1 \Big( 
        \inf\Big\{t > 0: \tsum_{i \leq p} N_i^*(p a_i / t) \leq p \Big\}
        + \sqrt{p} \big(\tsum_{i > p} a_i^2\big)^{1/2}
    \Big) \\
    & \leq \norm{\tsum_{i=1}^n a_i Z_i}_p
    \leq C_2 \Big( 
        \inf\Big\{t > 0: \tsum_{i \leq p} N_i^*(p a_i / t) \leq p \Big\}
        + \sqrt{p} \big(\tsum_{i > p} a_i^2\big)^{1/2}
    \Big),
\end{aligned}
\end{equation*}
where $C_1 = \min\{ \min\{\frac{\Exp[\abs{Z_i}]}{4}: i \in [n]\}, \frac{1}{4e}\}$ and $C_2 = 3 + 2 \max\{\kappa_i: i \in [n]\}$.
\end{lemma}

We apply \cref{thm:gluskin} to $\tsum_{i=1}^ n a'_i Z'_i$. First, we note that $C_1$ and $C_2$ are constants with only dependency on $\alpha$, because $\abs{Z'_i} \geq \min\{Y_i, Y_i^{2/\alpha}\}$ and $N_i(t) \geq \min\{x^2, x^\alpha\}$. Thus, \cref{thm:gluskin} implies that
\begin{equation*}
    \norm{Z^*}_p = \norm{\tsum_{i=1}^p a'_i Z'_i}_p
    \sim \inf\Big\{t > 0: \tsum_{i \leq p} N_i^*(p a_i \bar{L}_i / t) \leq p \Big\}
        + \sqrt{p} \big(\tsum_{i > p} a_i^2 \bar{L}_i^2\big)^{1/2},
\end{equation*}
as $a'_i = a_i \bar{L}_i$. Based on
\begin{equation*}
    N_i^*(p a_i \bar{L}_i/t) = \max\left\{
        \frac{p^2}{4t^2} a_i^2,
        (\alpha-1) \left(\frac{p a_i L_i}{\alpha t}\right)^\beta
    \right\},
\end{equation*}
we obtain
\begin{equation*}
\begin{aligned}
    \norm{Z^*}_p
    & \asymp \inf\Big\{t > 0: 
        \frac{p^2}{4t^2} \tsum_{i \leq p} a_i^2 
        + C(\alpha) \left(\frac{p}{t}\right)^\beta \tsum_{i \leq p} (a_i L_i)^\beta 
    \leq p \Big\} \\
    & \quad + \sqrt{p} \big(\tsum_{i > p} a_i^2 \bar{L}_i^2\big)^{1/2} \\
    & \asymp \sqrt{p} \left(\tsum_{i \leq p} a_i^2\right)^{1/2}
    + p^{1/\alpha} \left(\tsum_{i \leq p} (a_i L_i)^\beta\right)^{1/\beta} 
    + \sqrt{p} \big(\tsum_{i > p} a_i^2 \bar{L}_i^2\big)^{1/2},
\end{aligned}    
\end{equation*}
where $\beta = \frac{\alpha}{\alpha-1}$ and $\asymp$ indicates the equivalence with respect to constants depending only on $\alpha$, i.e., $x \asymp y \Leftrightarrow \frac{1}{C(\alpha)} x \leq y \leq C(\alpha) x$ for any $x, y \in \reals$. The desired result comes from the fact that 
\begin{equation*}
    \sqrt{p} \left(\tsum_{i \leq p} a_i^2\right)^{1/2}
    + p^{1/\alpha} \left(\tsum_{i \leq p} (a_i L_i)^\beta\right)^{1/\beta}
    \sim \sqrt{p} \left(\tsum_{i \leq p} a_i^2 \bar{L}_i^2 \right)^{1/2}
    + p^{1/\alpha} \left(\tsum_{i \leq p} (a_i L_i)^\beta\right)^{1/\beta},
\end{equation*}
for $\alpha \leq 2$.

\vspace{0.1in}
{\bf Cases of $\alpha \leq 1$:  } Using the relationship between $Z_i$ and $Y_i$ in \cref{thm:Z_by_Y}, we obtain \cref{thm:log_phi_bound}. Because $Y_i$ and $\abs{Y_i}^{2/\alpha}$ have log-concave and convex tails, they fall into Examples 3.2 and 3.3 of \cite{latala1997estimation}, respectively.

\begin{lemma} \label{thm:log_phi_bound}
    For $i = 1,\dots, n$ and $\eta > 0$,
    \begin{equation} \label{eq:log_phi_bound}
    \begin{aligned}
        & p \min\left\{ 1, 
        \max\left\{ \frac{\bar{L}_i^2}{\eta^2} p, 
        \frac{1}{C(\alpha)} \frac{L_i^p}{e^{2p} \eta^p} \left(\frac{p}{\alpha e}\right)^{\frac{p}{\alpha}}
        \right\}\right\} \\
        & \leq \log \phi_p\left( \frac{Z_i}{\eta} \right) 
        \leq C(\alpha) ~p~ \max\left\{ \frac{\bar{L}_i^2}{\eta^2} p, 
        \frac{e^{2p} L_i^p}{\eta^p} \left(\frac{p}{\alpha e}\right)^{\frac{p}{\alpha}}
        \right\},
    \end{aligned}\end{equation}
    where $\bar{L} = (\max\{1, L_i\})_{i=1,\dots,n}$.
\end{lemma}

For each $i \in [n]$, we replace $\eta$ in the above lemma with $\eta/a_i$ to obtain
\begin{equation*}\begin{aligned}
    & p \min\left\{ 1, 
    \frac{1}{C(\alpha)} \max\left\{ \frac{(a_i \bar{L}_i)^2}{\eta^2} p, 
    \frac{(a_i L_i)^p}{e^{2p} \eta^p} \left(\frac{p}{\alpha e}\right)^{\frac{p}{\alpha}}
    \right\}\right\} \\
    & \leq \log \phi_p\left( \frac{a_i Z_i}{\eta} \right) 
    \leq C(\alpha) ~p~ \max\left\{ \frac{(a_i \bar{L}_i)^2}{\eta^2} p, 
    \frac{e^{2p} (a_i L_i)^p}{\eta^p} \left(\frac{p}{\alpha e}\right)^{\frac{p}{\alpha}}
    \right\}.
\end{aligned}\end{equation*}
By the summation across $i \in [n]$, we bound $\sum_{i=1}^n \log \phi_p\left( \frac{a_i Z_i}{\eta} \right)$: for $\eta > 0$,
\begin{equation*}\begin{aligned}
    & p \sum_{i=1}^n \min\left\{ 1, 
    \frac{1}{C(\alpha)} \max\left\{ \frac{(a_i \bar{L}_i)^2}{\eta^2} p, 
    \frac{(a_i L_i)^p}{e^{2p} \eta^p} \left(\frac{p}{\alpha e}\right)^{\frac{p}{\alpha}}
    \right\}\right\} \\
    & \leq \sum_{i=1}^n \log \phi_p\left( \frac{a_i Z_i}{\eta} \right) 
    \leq C(\alpha) ~p~ \sum_{i=1}^n \max\left\{ \frac{(a_i \bar{L}_i)^2}{\eta^2} p, 
    \frac{e^{2p} (a_i L_i)^p}{\eta^p} \left(\frac{p}{\alpha e}\right)^{\frac{p}{\alpha}}
    \right\}.
\end{aligned}\end{equation*}
For the upperbound,
\begin{equation*}
\begin{aligned}
    & C(\alpha) ~p~ \sum_{i=1}^n \max\left\{ \frac{(a_i \bar{L}_i)^2}{\eta^2} p, 
    \frac{e^{2p} (a_i L_i)^p}{\eta^p} \left(\frac{p}{\alpha e}\right)^{\frac{p}{\alpha}}
    \right\} \\
    & \leq C(\alpha) ~p~ \sum_{i=1}^n \left( \frac{(a_i \bar{L}_i)^2}{\eta^2} p +  
    \frac{e^{2p} (a_i L_i)^p}{\eta^p} \left(\frac{p}{\alpha e}\right)^{\frac{p}{\alpha}}
    \right) \\
    & \leq 2 C(\alpha) ~p~ \max\left\{ \sum_{i=1}^n \frac{(a_i \bar{L}_i)^2}{\eta^2} p, 
    \sum_{i=1}^n \frac{e^{2p} (a_i L_i)^p}{\eta^p} \left(\frac{p}{\alpha e}\right)^{\frac{p}{\alpha}}
    \right\} \\
    & = 2 C(\alpha) ~p~ \max\left\{ 
        \frac{\norm{a \odot \bar{L}}_2^2}{\eta^2} p, 
        \frac{e^{2p} \norm{a \odot L}_p^p}{\eta^p} \left(\frac{p}{\alpha e}\right)^{\frac{p}{\alpha}}
    \right\}.
\end{aligned}
\end{equation*}
For the lowerbound, if $\frac{1}{C(\alpha)} \max\left\{ \frac{(a_i \bar{L}_i)^2}{\eta^2} p, 
\frac{(a_i L_i)^p}{e^{2p} \eta^p} \left(\frac{p}{\alpha e}\right)^{\frac{p}{\alpha}}
\right\} \leq 1$ for all $i \in [n]$, then
\begin{equation*}\begin{aligned}
    & p \sum_{i=1}^n \min\left\{ 1, 
    \frac{1}{C(\alpha)} \max\left\{ \frac{(a_i \bar{L}_i)^2}{\eta^2} p, 
    \frac{(a_i L_i)^p}{e^{2p} \eta^p} \left(\frac{p}{\alpha e}\right)^{\frac{p}{\alpha}}
    \right\}\right\} \\
    & = p \sum_{i=1}^n 
    \frac{1}{C(\alpha)} \max\left\{ \frac{(a_i \bar{L}_i)^2}{\eta^2} p, 
    \frac{(a_i L_i)^p}{e^{2p} \eta^p} \left(\frac{p}{\alpha e}\right)^{\frac{p}{\alpha}}
    \right\} \\
    & \geq \frac{p}{C(\alpha)}  
    \max\left\{ \sum_{i=1}^n \frac{(a_i \bar{L}_i)^2}{\eta^2} p, 
    \sum_{i=1}^n \frac{(a_i L_i)^p}{e^{2p} \eta^p} \left(\frac{p}{\alpha e}\right)^{\frac{p}{\alpha}}
    \right\} \\
    & = \frac{p}{C(\alpha)} \max\left\{ 
        \frac{\norm{a \odot \bar{L}}_2^2}{\eta^2} p, 
        \frac{e^{2p} \norm{a \odot L}_p^p}{\eta^p} \left(\frac{p}{\alpha e}\right)^{\frac{p}{\alpha}}
    \right\}.
\end{aligned}\end{equation*}
On the other hand, without loss of generality, if $\frac{1}{C(\alpha)} \max\left\{ 
\frac{(a_1 \bar{L}_1)^2}{\eta^2} p, 
\frac{(a_1 L_1)^p}{e^{2p} \eta^p} \left(\frac{p}{\alpha e}\right)^{\frac{p}{\alpha}}
\right\} > 1$, then
\begin{equation*}\begin{aligned}
    & p \sum_{i=1}^n \min\left\{ 1, 
    \frac{1}{C(\alpha)} \max\left\{ \frac{(a_i \bar{L}_i)^2}{\eta^2} p, 
    \frac{(a_i L_i)^p}{e^{2p} \eta^p} \left(\frac{p}{\alpha e}\right)^{\frac{p}{\alpha}}
    \right\}\right\} \\
    & \geq p \min\left\{ 1, 
    \frac{1}{C(\alpha)} \max\left\{ 
        \frac{(a_1 \bar{L}_1)^2}{\eta^2} p, 
        \frac{(a_1 L_1)^p}{e^{2p} \eta^p} \left(\frac{p}{\alpha e}\right)^{\frac{p}{\alpha}}
    \right\}\right\} \\
    & = p.
\end{aligned}\end{equation*}
In sum,
\begin{equation}\begin{aligned} \label{eq:log_phi_sum_bound}
    & p \min\left\{ 1, \frac{1}{C(\alpha)} \max\left\{ 
        \frac{\norm{a \odot \bar{L}}_2^2}{\eta^2} p, 
        \frac{\norm{a \odot L}_p^p}{e^{2p} \eta^p} \left(\frac{p}{\alpha e}\right)^{\frac{p}{\alpha}}
        \right\}\right\} \\
    & \leq \sum_{i=1}^n \log \phi_p\left(\frac{a_i Z_i}{\eta}\right)
    \leq C(\alpha) ~p ~\max\left\{ 
        \frac{\norm{a \odot \bar{L}}_2^2}{\eta^2} p, 
        \frac{e^{2p} \norm{a \odot L}_p^p}{\eta^p} \left(\frac{p}{\alpha e}\right)^{\frac{p}{\alpha}}
        \right\},
\end{aligned}\end{equation}
where $a \odot L = (a_i L_i)_{i=1,\dots,n}$. 
    
Now we give a tight bound for $\mnorm{(a_iZ_i)}_p$ according to its definition in \cref{eq:sequence_Orlicz_norm}. 
For the upperbound, suppose that $\eta_0 = C_1(\alpha) \max\left\{\sqrt{p} \norm{a \odot \bar{L}}_2, p^{1/\alpha} \norm{a \odot L}_p \right\}$. By the upperbound in \cref{eq:log_phi_sum_bound},
\begin{equation*}\begin{aligned} 
    & \sum_{i=1}^n \log \phi_p\left(\frac{a_i Z_i}{\eta_0}\right)
    \leq C(\alpha) ~p ~\max\left\{ 
        \frac{\norm{a \odot \bar{L}}_2^2}{\eta_0^2} p, 
        \frac{e^{2p} \norm{a \odot L}_p^p}{\eta_0^p} \left(\frac{p}{\alpha e}\right)^{\frac{p}{\alpha}}
        \right\} \\
    & \leq C(\alpha) ~p ~\max\left\{ 
        \frac{\norm{a \odot \bar{L}}_2^2}{(C_1(\alpha) \sqrt{p} \norm{a \odot \bar{L}}_2)^2} p, 
        \frac{e^{2p} \norm{a \odot L}_p^p}{(C_1(\alpha) p^{1/\alpha} \norm{a \odot L}_p)^p} \left(\frac{p}{\alpha e}\right)^{\frac{p}{\alpha}}
        \right\} \\
    & \leq C(\alpha) ~p ~\max\left\{ \frac{1}{C_1(\alpha)^2}, \frac{e^{2p}}{C_1(\alpha)^p} \left(\frac{1}{\alpha e}\right)^{\frac{p}{\alpha}} \right\}.
\end{aligned}\end{equation*}
Thus, for $C_1(\alpha) \geq \sqrt{2\max\{C(\alpha),1\}} \max\left\{1, \frac{e^2}{(\alpha e)^{1/\alpha}} \right\}$, $\sum_{i=1}^n \log \phi_p\left(\frac{a_i Z_i}{\eta_0}\right) \leq \frac{p}{2}$, and 
$$\mnorm{(X_i)}_p \leq \eta_0 = C_1(\alpha) \max\left\{\sqrt{p} \norm{a \odot \bar{L}}_2, p^{1/\alpha} \norm{a \odot L}_p \right\}.$$
For the lowerbound, suppose that $\eta_0 = \frac{1}{C_2(\alpha)} \max\left\{\sqrt{p} \norm{a \odot \bar{L}}_2, p^{1/\alpha} \norm{a \odot L}_p \right\}$. If $\sqrt{p} \norm{a \odot \bar{L}}_2 \leq p^{1/\alpha} \norm{a \odot L}_p$, by the lowerbound in \cref{eq:log_phi_sum_bound},
\begin{equation*}\begin{aligned} 
    & \sum_{i=1}^n \log \phi_p\left(\frac{a_i Z_i}{\eta_0}\right)
    \geq p \min\left\{ 1, \frac{1}{C(\alpha)} \max\left\{ 
        \frac{\norm{a \odot \bar{L}}_2^2}{\eta_0^2} p, 
        \frac{\norm{a \odot L}_p^p}{e^{2p} \eta_0^p} \left(\frac{p}{\alpha e}\right)^{\frac{p}{\alpha}}
        \right\}\right\} \\
    & \geq p \min\left\{ 1, \frac{1}{C(\alpha)} 
        \frac{C_2(\alpha)^p \norm{a \odot L}_p^p}{e^{2p} (p^{1/\alpha} \norm{a \odot L}_p)^p} \left(\frac{p}{\alpha e}\right)^{\frac{p}{\alpha}}
        \right\}
    \geq p \min\left\{ 1, \frac{C_2(\alpha)^p}{C(\alpha)} 
        \left(\frac{1}{\alpha e^{2\alpha+1}}\right)^{\frac{p}{\alpha}}
        \right\}.
\end{aligned}\end{equation*}
If $\sqrt{p} \norm{a \odot \bar{L}}_2 \geq p^{1/\alpha} \norm{a \odot L}_p$, 
\begin{equation*}\begin{aligned} 
    & \sum_{i=1}^n \log \phi_p\left(\frac{a_i Z_i}{\eta_0}\right)
    \geq p \min\left\{ 1, \frac{1}{C(\alpha)} \max\left\{ 
        \frac{\norm{a \odot \bar{L}}_2^2}{\eta_0^2} p, 
        \frac{\norm{a \odot L}_p^p}{e^{2p} \eta_0^p} \left(\frac{p}{\alpha e}\right)^{\frac{p}{\alpha}}
        \right\}\right\} \\
    & \geq p \min\left\{ 1, \frac{1}{C(\alpha)}
        \frac{C_2(\alpha)^2 \norm{a \odot \bar{L}}_2^2}{(\sqrt{p} \norm{a \odot \bar{L}}_2)^2} p
        \right\}
    \geq p \min\left\{ 1, \frac{C_2(\alpha)^2}{C(\alpha)} 
        \right\}.
\end{aligned}\end{equation*}
In sum,
\begin{equation*}\begin{aligned} 
    & \sum_{i=1}^n \log \phi_p\left(\frac{a_i Z_i}{\eta}\right)
    \geq p \min\left\{ 1, \frac{C_2(\alpha)^2}{C(\alpha)}, \frac{C_2(\alpha)^p}{C(\alpha)} 
        \left(\frac{1}{\alpha e^{2\alpha+1}}\right)^{\frac{p}{\alpha}}
        \right\} \geq p,
\end{aligned}\end{equation*}
for $C_2(\alpha) \geq \sqrt{\max\{C(\alpha),1\}} (\alpha e^{2\alpha+1})^{1/\alpha}$. Because $\sum_{i=1}^n \log \phi_p\left(\frac{a_i Z_i}{\eta}\right)$ strictly decreases for $\eta \in (0, \infty)$,
$$\mnorm{(a_iZ_i)}_p \geq \eta_0 = \frac{1}{C_2(\alpha)} \max\left\{\sqrt{p} \norm{a \odot \bar{L}}_2, p^{1/\alpha} \norm{a \odot L}_p \right\}.$$
The resulting bound for $\mnorm{(X_i)}_p$ is 
\begin{equation*}\begin{aligned}
    & \frac{1}{C(\alpha)} \max\left\{\sqrt{p} \norm{a \odot \bar{L}}_2,
    p^{1/\alpha} \norm{a \odot L}_p \right\} \\
    & \leq \mnorm{(a_i Z_i)}_p 
    \leq C(\alpha) \max\left\{\sqrt{p} \norm{a \odot \bar{L}}_2,
    p^{1/\alpha} \norm{a \odot L}_p \right\}.
\end{aligned}\end{equation*}
Then, \cref{thm:moment_bound_by_Orlicz_norm} derives 
\begin{equation*}\begin{aligned}
    & \frac{1}{C(\alpha)} \max\left\{\sqrt{p} \norm{a \odot \bar{L}}_2,
    p^{1/\alpha} \norm{a \odot L}_p \right\} \\
    & \leq \norm{Z^*}_p 
    \leq C(\alpha) \max\left\{\sqrt{p} \norm{a \odot \bar{L}}_2,
    p^{1/\alpha} \norm{a \odot L}_p \right\},
\end{aligned}\end{equation*}
for $p \geq 2$, which implies the lowerbound in \cref{eq:moment_lower}.
We bound $\norm{a \odot L}_p$ in terms of $\norm{a \odot L}_2$ and $\norm{a \odot L}_\infty$ following the argument in Corollary 1.2 in \cite{bogucki2015suprema}. In the proof, \cite{bogucki2015suprema} showed that for $p \geq 2$,
\begin{equation*}
    \norm{a \odot L}_\infty \leq \norm{a \odot L}_p \leq 
    \frac{e^{(2-\alpha)/\alpha}}{p^{1/\alpha}} (\sqrt{p} \norm{a \odot L}_2 + p^{1/\alpha} \norm{a \odot L}_\infty).
\end{equation*}
Plugging it in \cref{eq:moment_lower,eq:moment_upper}, we obtain the desired results for $p \geq 2$.
Now the remaining cases are for $1 \leq p < 2$. For the upper bound,
\begin{equation*}\begin{aligned}
    \norm{Z^*}_p 
    & \leq \norm{Z^*}_2
    \leq C(\alpha) \max\left\{\sqrt{2} \norm{a \odot \bar{L}}_2,
    2^{1/\alpha} \norm{a \odot L}_\infty \right\} \\
    & \leq C(\alpha) \max\left\{\sqrt{p} \norm{a \odot \bar{L}}_2,
    p^{1/\alpha} \norm{a \odot L}_\infty \right\}.
\end{aligned}\end{equation*}
On the other hand, because 
\begin{equation*}
\begin{aligned}
    & \norm{Z^*}_2 
    \geq \frac{1}{C(\alpha)} \max\left\{\sqrt{2} \norm{a \odot \bar{L}}_2,
    2^{1/\alpha} \norm{a \odot L}_\infty \right\}  \\
    & \norm{Z^*}_4
    \leq C(\alpha) \max\left\{2 \norm{a \odot \bar{L}}_2,
    4^{1/\alpha} \norm{a \odot L}_\infty \right\},
\end{aligned}
\end{equation*}
by the Paley-Zygmund inequality,
\begin{equation*}\begin{aligned}
    \Pr\left[\abs{Z^*} \geq \frac{1}{2} \norm{Z^*}_2\right]
    & \geq \left(1 - \frac{1}{4}\frac{\norm{Z^*}_2^2}{\norm{Z^*}_2^2}\right)^2 \frac{\norm{Z^*}_2^4}{\norm{Z^*}_{4}^{4}} \\
    & \geq \frac{3}{4} \frac{1}{C(\alpha)} \left(\frac{\max\left\{
        \sqrt{2} \norm{a \odot \bar{L}}_2,
        2^{1/\alpha} \norm{a \odot L}_\infty 
    \right\}}
    {\max\left\{
        2 \norm{a \odot \bar{L}}_2,
        4^{1/\alpha} \norm{a \odot L}_\infty 
    \right\}}\right)^4 \\
    & \geq \frac{1}{C(\alpha)} \left(\max\left\{
        \frac{\sqrt{2} \norm{a \odot \bar{L}}_2}{2 \norm{a \odot \bar{L}}_2},
        \frac{2^{1/\alpha} \norm{a \odot L}_\infty}{4^{1/\alpha} \norm{a \odot L}_\infty}
    \right\}\right)^4 \\
    & \geq \frac{1}{C(\alpha)}.
\end{aligned}\end{equation*}
Hence, for $1 \leq p < 2$, by the Markov inequality,
\begin{equation*}
\begin{aligned}
    \norm{Z^*}_p & \geq \norm{Z^*}_1 \geq \frac{1}{2} \norm{Z^*}_2 \cdot \Pr\left[ \abs{Z^*} \geq \frac{1}{2}\norm{Z^*}_2 \right] \\
    & \geq \frac{1}{2 C(\alpha)} \norm{Z^*}_2 \\
    & \geq \frac{1}{C(\alpha)} \max\left\{\norm{a \odot \bar{L}}_2, \norm{a \odot L}_\infty \right\} \\
    & \geq \frac{1}{C(\alpha)} \max\left\{\sqrt{p} \norm{a \odot \bar{L}}_2, p^{1/\alpha} \norm{a \odot L}_\infty \right\}.
\end{aligned}
\end{equation*}

\end{proof}

\begin{proof}[Proof of Theorem~\ref{thm:GBO_bound}]

For the upperbound, we use the moment bound in \cref{thm:moment_bound} using Proposition C.1 in \cite{kuchibhotla2018moving}:
\begin{lemma}[Proposition C.1, \cite{kuchibhotla2018moving}] \label{thm:moment_bound_to_GBO_bound}
    If $\norm{X}_p \leq C_1 \sqrt{p} + C_2 p^{1/\alpha}$, holds for $p \geq 1$ and some constants $C_1$ and $C_2$, then $\norm{X}_{\phi_{\alpha, L}} \leq 2 e C_1$ where $L := 4^{1/\alpha} \frac{C_2}{2 C_1}$.
\end{lemma}
Based on \cref{thm:moment_bound},
\begin{equation*}
    \norm{X^*}_{\phi_{\alpha, L^*}} \leq C(\alpha) \norm{a \odot \bar{L}}_2 
\end{equation*}
where $L^* = C(\alpha) \norm{a \odot L}_\beta/\norm{a \odot \bar{L}}_2$.
For the lowerbound, we use Proposition A.4 in \cite{kuchibhotla2018moving}:
\begin{lemma}[Proposition A.4, \cite{kuchibhotla2018moving}]
    For any random variable $X$,
    \begin{equation*}
        \frac{1}{C(\alpha)} \sup_{p \geq 1} \frac{\norm{X}_p}{\sqrt{p} + L p^{1/\alpha}} 
        \leq \norm{X}_{\phi_{\alpha, L}}
        \leq C(\alpha) \sup_{p \geq 1} \frac{\norm{X}_p}{\sqrt{p} + L p^{1/\alpha}}.
    \end{equation*}
\end{lemma}
Because \cref{eq:moment_lower} implies that
\begin{equation*}
\begin{aligned}
    \sup_{p \geq 1} \frac{\norm{Z^*}_p}{\sqrt{p} + L^* p^{1/\alpha}} 
    & \geq \frac{\norm{Z^*}_n}{\sqrt{n} + L^* n^{1/\alpha}} \\
    & \geq \frac{1}{C(\alpha)} 
    \frac{\max\left\{\sqrt{n} \norm{a \odot \bar{L}}_2,
        n^{1/\alpha} \norm{a \odot L}_\beta \right\}}{\sqrt{n} + L^* n^{1/\alpha}} \\
    & \geq \frac{1}{C(\alpha)} \norm{a \odot \bar{L}}_2,
\end{aligned}
\end{equation*}
the desired lower bound $\norm{Z^*}_{\phi_{\alpha, L^*}} \geq \frac{1}{C(\alpha)} \norm{a \odot \bar{L}}_2$ is achieved.

\end{proof}

\begin{proof}[Proof of Theorem~\ref{thm:tail_bound}]

For \cref{eq:tail_prob_upper_logconcave,eq:tail_prob_lower_logconcave}, see the proof of Corollary in \cite{gluskin1995tail}. \cref{eq:tail_prob_upper} is the result of Chernoff's inequliaty on the GBO norm upperbound in \cref{eq:GBO_upper}. For \cref{eq:tail_prob_lower}, we use the Paley-Zygmund inequality \cite{paley1932note} based on the tight moment bound established in \cref{thm:moment_bound}. For $p \geq 2$,
\begin{equation*}\begin{aligned}
    \Pr\left[\abs{Z^*} \geq \frac{1}{2} \norm{Z^*}_p\right]
    & = \Pr\left[\abs{Z^*}^p \geq \frac{1}{2^p} \norm{Z^*}_p^p \right] \\
    & \geq \left(1 - \frac{1}{2^p}\right)^2 \frac{\norm{Z^*}_p^{2p}}{\norm{Z^*}_{2p}^{2p}} \\
    & \geq \frac{1}{2 C(\alpha)^p} 
        \frac{\max\{p^{p}\norm{a \odot \bar{L}}_2^{2p},
            p^{2p/\alpha}\norm{a \odot L}_p^{2p}\}}
            {\max\{(2p)^{p}\norm{a \odot \bar{L}}_2^{2p},
            (2p)^{2p/\alpha}\norm{a \odot L}_{2p}^{2p}\}} \\
    & \geq \frac{1}{2 C(\alpha)^p} 
        \max\left\{ \frac{p^{p}\norm{a \odot \bar{L}}_2^{2p}}
            {(2p)^{p}\norm{a \odot \bar{L}}_2^{2p}}, 
        \frac{p^{2p/\alpha}\norm{a \odot L}_p^{2p}}
            {(2p)^{2p/\alpha}\norm{a \odot L}_{2p}^{2p}} \right\} \\
    & \geq \frac{1}{2} \exp( - C(\alpha) ~p ).
\end{aligned}\end{equation*}
Plugging the moment lowerbound (\cref{eq:moment_lower}) in,
\begin{equation*}
    \Pr\left[\abs{Z^*} \geq \frac{1}{C(\alpha)}\max\{\sqrt{p} \norm{a \odot \bar{L}}_2,
            p^{1/\alpha}\norm{a \odot L}_\infty\} \right]
    \geq \frac{1}{2} \exp( - C(\alpha) ~p ).
\end{equation*}
With $p = \min\left\{\left(\frac{2 C(\alpha) t}{\norm{a \odot \bar{L}}_2}\right)^2, \left(\frac{2 C(\alpha) t}{\norm{a \odot L}_\infty}\right)^\alpha\right\}$, we obtain the desired lowerbound for $t \geq \frac{1}{2} \norm{Z^*}_2$. Last, for $t < \frac{1}{2}\norm{Z^*}_2$, we use the Paley-Zygmund inequality on the second moment bound:
\begin{equation*}\begin{aligned}
    \Pr\left[\abs{Z^*} \geq t\right]
    & \geq \left(1 - \frac{t^2}{\norm{Z^*}_2^2}\right)^2 \frac{\norm{Z^*}_2^4}{\norm{Z^*}_{4}^{4}} \\
    & \geq \frac{1}{C(\alpha)} \exp\left(-\frac{2t^2}{\norm{Z^*}_2^2}\right) \\
    & \geq \frac{1}{C(\alpha)} \exp\left(-\frac{1}{C(\alpha)} \min\left\{\left(\frac{2t}{\norm{a \odot \bar{L}}_2}\right)^2, \left(\frac{2t}{\norm{a \odot L}_\infty}\right)^\alpha\right\}\right). \\
\end{aligned}\end{equation*}

\end{proof}

\subsection{Proof for Application} \label{sec:pf_tail_bound_sigma_sse}

\begin{proof}[Proof of Theorem~\ref{thm:tail_bound_sigma_sse}]

Following the same argument in the proof of \cref{thm:moment_bound} (see \cref{eq:prob_bound_by_Z}), 
\begin{equation*}\begin{aligned}
    \Pr[\max\{(\hat\Sigma_{ij}^{(l)} - \Sigma_{ij}^{(l)})^2 - \mu, 0\} \geq t ] 
    & \leq \exp(- C \min\{n t, n \sqrt{t}\}), \\
\end{aligned}\end{equation*}
where $\mu = \frac{C \log 2}{n}$. We note that based on the same argument in \cref{thm:Z_by_Y}, 
\begin{equation*}\begin{aligned}
    \Pr\left[C \max\left\{\frac{Y^{(l)2}}{n}, \frac{Y^{(l)4}}{n^2} \right\} \geq t \right] 
    &  = \exp(- C \min\{n t, n \sqrt{t}\}), \\
\end{aligned}\end{equation*}
where $Y^{(l)}$ is a sub-Gaussian random variable with the same survival function as in \cref{eq:Y}, and therefore, $(\hat\Sigma_{ij}^{(l)} - \Sigma_{ij}^{(l)})^2 
\overset{d}{\leq} C \max\left\{\frac{Y^{(l)2}}{n}, \frac{Y^{(l)4}}{n^2} \right\} + \mu 
\leq C\frac{Y^{(l)2}}{n} + C\frac{Y^{(l)4}}{n^2} + \mu$. Because every term in both hand sides is almost surely positive for every $l \in [m]$,
\begin{equation*}
\begin{aligned}
    \Pr\left[ \sum_{l=1}^m (\hat\Sigma_{ij}^{(l)} - \Sigma_{ij}^{(l)})^2 \geq t + C \frac{m}{n} \right]
    & \leq \Pr\left[ C \sum_{l=1}^m \frac{Y^{(l)2}}{n} + C \sum_{l=1}^m \frac{Y^{(l)4}}{n^2}  \geq t \right] \\
    & \leq \Pr\left[ C \sum_{l=1}^m \frac{Y^{(l)2}}{n} \geq \frac{t}{2} \right] 
    + \Pr\left[ C \sum_{l=1}^m \frac{Y^{(l)4}}{n^2} \geq \frac{t}{2} \right].
\end{aligned}
\end{equation*}
Based on our observation that $\norm{Y^{(l)2}}_{\psi_1}, \norm{Y^{(l)4}}_{\psi_{1/2}} \leq C$, Theorem 3.1 of \cite{kuchibhotla2018moving} derives
\begin{equation*}
    \norm*{\sum_{l=1}^m \frac{Y^{(l)2}}{n}}_{\phi_{1,\frac{1}{\sqrt{m}}}} \leq C \frac{\sqrt{m}}{n},
    \quad \norm*{\sum_{l=1}^m \frac{Y^{(l)4}}{n^2}}_{\phi_{\frac{1}{2},\frac{1}{\sqrt{m}}}} \leq C \frac{\sqrt{m}}{n^2}.
\end{equation*}
In sum,
\begin{equation} 
\begin{aligned} 
    & \Pr\left[ \sum_{l=1}^m (\hat\Sigma_{ij}^{(l)} - \Sigma_{ij}^{(l)})^2 \geq t + C \frac{m}{n} \right]
    \leq \Pr\left[ C \sum_{l=1}^m \frac{Y^{(l)2}}{n} \geq \frac{t}{2} \right] 
    + \Pr\left[ C \sum_{l=1}^m \frac{Y^{(l)4}}{n^2} \geq \frac{t}{2} \right] \\
    & \leq 2 \exp\left(- C \min\left\{\frac{n^2t^2}{m}, nt\right\} \right) 
    + 2 \exp\left(- C \min\left\{\frac{n^4t^2}{m}, n\sqrt{t}\right\} \right) \\
    & \leq 4 \exp\left(- C \min\left\{\frac{n^2t^2}{m}, nt, n\sqrt{t} \right\} \right),
\end{aligned}
\end{equation}
and therefore, $\sum_{l=1}^m (\hat\Sigma_{ij}^{(l)} - \Sigma_{ij}^{(l)})^2 \leq C \left( \frac{m + \sqrt{m\nu} + \nu}{n} + \frac{\nu^2}{n^2} \right)$ with probability at least $1 - e^{-\nu}$. Thus, if $\frac{\log q}{n} \leq C$, for $\nu = C \log(qn)$, 
\begin{equation*}
    \sup_{i,j \in [q]} \sum_{l=1}^m (\hat\Sigma_{ij}^{(l)} - \Sigma_{ij}^{(l)})^2 
    \leq C \frac{m + \log(qn)}{n},
\end{equation*}
with probability at least $1 - \frac{1}{n}$. 

Now we prove the lower bound.  
Suppose that $\Pr[\abs{\hat\Sigma_{ij}^{(l)} - \Sigma_{ij}^{(l)}} \geq t] {\geq} \exp( - C \min\{ nt, nt^2\} )$ for some $C > 0$. Then, 
\begin{equation*}
\begin{aligned}
    \sum_{l=1}^m (\hat\Sigma_{ij}^{(l)} - \Sigma_{ij}^{(l)})^2
    & \overset{d}{\geq} C \sum_{l=1}^m \max\left\{ \frac{Y^{(l)2}}{n}, \frac{Y^{(l)4}}{n^2} \right\} \\
    & \geq C \max\left\{ \sum_{l=1}^m \frac{Y^{(l)2}}{n}, \sum_{l=1}^m \frac{Y^{(l)4}}{n^2} \right\},
\end{aligned}    
\end{equation*}
which means that
\begin{equation*}
\begin{aligned}
    \Pr\left[\sum_{l=1}^m (\hat\Sigma_{ij}^{(l)} - \Sigma_{ij}^{(l)})^2 \geq t \right]
    & \geq \Pr\left[ C \max\left\{ \sum_{l=1}^m \frac{Y^{(l)2}}{n}, \sum_{l=1}^m \frac{Y^{(l)4}}{n^2} \right\} \geq t\right] \\
    & \geq \max\left\{ \Pr\left[ C\sum_{l=1}^m \frac{Y^{(l)2}}{n} \geq t \right], \Pr\left[ C \sum_{l=1}^m \frac{Y^{(l)4}}{n^2} \geq t\right]\right\}.
\end{aligned}    
\end{equation*}
Applying \cref{thm:tail_bound} to each of $\Pr\left[ C\sum_{l=1}^m \frac{Y^{(l)2}}{n} \geq t \right]$ and $\Pr\left[ C \sum_{l=1}^m \frac{Y^{(l)4}}{n^2} \geq t\right]$, we prove \cref{eq:lower_bound_sigma_sse}. 

\end{proof}

\subsection{Proof of Lemmas} \label{sec:pf_lemmas}

\begin{proof}[Proof of Lemma~\ref{thm:Z_by_Y}]

Because $\alpha \leq 1$, for $y \in \reals$, $L_i y^{2/\alpha} \geq y  \Longleftrightarrow y \geq L_i^{-\frac{\alpha}{2-\alpha}}$. Hence,
\begin{equation*}
    \Pr[\max\{\abs{Y_i}, L_i \abs{Y_i}^{2/\alpha}\} \geq t] 
    = \Pr[t \leq \abs{Y_i} \leq L_i^{-\frac{\alpha}{2-\alpha}}]
    + \Pr[L_i \abs{Y_i}^{2/\alpha} \geq t, \abs{Y_i} > L_i^{-\frac{\alpha}{2-\alpha}}].
\end{equation*}

If $t \leq L_i^{-\frac{\alpha}{2-\alpha}}$, then
\begin{equation*}\begin{aligned}
    \Pr[L_i \abs{Y_i}^{2/\alpha} \geq t, \abs{Y_i} > L_i^{-\frac{\alpha}{2-\alpha}}]
    & = \Pr[\abs{Y_i} \geq (t/L_i)^{\alpha/2}, \abs{Y_i} > L_i^{-\frac{\alpha}{2-\alpha}}] \\
    & = \Pr[\abs{Y_i} \geq \max\{(t/L_i)^{\alpha/2}, L_i^{-\frac{\alpha}{2-\alpha}}\}] \\
    & = \Pr[\abs{Y_i} \geq L_i^{-\frac{\alpha}{2-\alpha}}], 
\end{aligned}\end{equation*}
\begin{equation*}\begin{aligned}
    \Pr[\max\{\abs{Y_i}, L_i \abs{Y_i}^{2/\alpha}\} \geq t] 
    & = \Pr[t \leq \abs{Y_i} \leq L_i^{-\frac{\alpha}{2-\alpha}}]
      + \Pr[\abs{Y_i} \geq L_i^{-\frac{\alpha}{2-\alpha}}] \\
    & = \Pr[\abs{Y_i} \geq t] = e^{-t^2}.
\end{aligned}\end{equation*}

If $t > L_i^{-\frac{\alpha}{2-\alpha}}$, then $\Pr[\abs{Y_i} \geq t, \abs{Y_i} \leq L_i^{-\frac{\alpha}{2-\alpha}}] = 0$,
\begin{equation*}\begin{aligned}
    \Pr[L_i \abs{Y_i}^{2/\alpha} \geq t, \abs{Y_i} > L_i^{-\frac{\alpha}{2-\alpha}}]
    & = \Pr[\abs{Y_i} \geq \max\{(t/L_i)^{\alpha/2}, L_i^{-\frac{\alpha}{2-\alpha}}\}] \\
    & = \Pr[\abs{Y_i} \geq (t/L_i)^{\alpha/2}], 
\end{aligned}\end{equation*}
\begin{equation*}\begin{aligned}
    \Pr[\max\{\abs{Y_i}, L_i \abs{Y_i}^{2/\alpha}\} \geq t] 
    & = \Pr[t \leq \abs{Y_i} \leq L_i^{-\frac{\alpha}{2-\alpha}}]
      + \Pr[\abs{Y_i} \geq L_i^{-\frac{\alpha}{2-\alpha}}] \\
    & = \Pr[\abs{Y_i} \geq (t/L_i)^{\alpha/2}] = e^{-(t/L_i)^\alpha}.
\end{aligned}\end{equation*}

In sum, 
\begin{equation*}\begin{aligned}
    \Pr[\max\{\abs{Y_i}, L_i \abs{Y_i}^{2/\alpha}\} \geq t] 
    & = \begin{cases}
        e^{-t^2}, & t \leq L_i^{-\frac{\alpha}{2-\alpha}} \\
        e^{-(t/L_i)^\alpha}, & t > L_i^{-\frac{\alpha}{2-\alpha}}
    \end{cases} \\
    & = \exp(-\min\{t^2, (t/L_i)^\alpha\}),
\end{aligned}\end{equation*}
which is the same as the survival function of $X_i$.

\end{proof}

\begin{proof}[Proof of Lemma~\ref{thm:gluskin}]

By Karamata's inequality (Proposition B.4.a in Chapter 16 of \cite{marshall1979inequalities}) and the definition of $\kappa_i$,
\begin{equation} \label{eq:moment_i>p_upper_bound}
\begin{aligned}
    \Exp\left[\abs*{\tsum_{i>p} a_i Z_i}^p\right]^{1/p}
    & \leq \Exp\left[\abs*{\tsum_{i>p} \kappa_i a_i \epsilon_i Y_i^2}^p\right]^{1/p} \\
    & \leq \max\{\kappa_i: i \in [n]\} \Exp\left[\abs*{\tsum_{i>p} a_i \epsilon_i Y_i^2}^p\right]^{1/p} \\
    & \leq \max\{\kappa_i: i \in [n]\} \max\{ \sqrt{p} (\tsum_{i>p} a_i^2)^{1/2}, p a_{\ceil{p}} \},
\end{aligned}
\end{equation}
where $\epsilon_i$ is a Rademacher random variable, and $\ceil{p}$ is the smallest integer larger than $p$.
If we assume that $\inf\Big\{t > 0: \tsum_{i \leq p} N_i^*(p a_i / t) \leq p \Big\} + \sqrt{p} \big(\tsum_{i > p} a_i^2\big)^{1/2} \leq 1$, then
\begin{equation*}
    \inf\Big\{t > 0: \tsum_{i \leq p} N_i^*(p a_i / t) \leq p \Big\} \leq 1,
\end{equation*}
and by recalling the definition of $N_{i}^*(t) = \sup\{st - N_{i}(s): s > 0\}$,
\begin{equation*}
    1 
    \geq \frac{1}{p} \tsum_{i \leq p} N_i^*(p a_i)
    \geq \frac{\floor{p}}{p} N_{\floor{p}}^*(p a_{\floor{p}})
    \geq \floor{p}\left( a_{\floor{p}} - \frac{N_{\floor{p}}(1)}{p} \right)
    = \floor{p}\left( a_{\floor{p}} - \frac{1}{p} \right).
\end{equation*}
Hence, $a_i \leq \frac{2}{p}$ for all $i > p$, and therefore \cref{eq:moment_i>p_upper_bound} results in
\begin{equation*}
\begin{aligned}
    \Exp\left[\abs*{\tsum_{i>p} a_i Z_i}^p\right]^{1/p} \leq 2 \max\{\kappa_i: i \in [n]\}.
\end{aligned}
\end{equation*}
By the homogeneity of the $p$-norm, 
\begin{equation*}
\begin{aligned}
    & \Exp\left[\abs*{\tsum_{i>p} a_i Z_i}^p\right]^{1/p} \\
    & \leq 2 \max\{\kappa_i: i \in [n]\} \left( \inf\Big\{t > 0: \tsum_{i \leq p} N_i^*(p a_i / t) \leq p \Big\} + \sqrt{p} \big(\tsum_{i > p} a_i^2\big)^{1/2} \right).
\end{aligned}
\end{equation*}
On the other hand, because $Z_i \sim \epsilon_i N_i^{-1}(\abs{Y_i^2})$,
\begin{equation*}
    \Exp\left[ \tsum_{i \leq p} a_i Z_i \right]
    = \Exp\left[ \tsum_{i \leq p} a_i \epsilon_i N_i^{-1}(\abs{Y_i^2}) \right].
\end{equation*}
Due to the definition of $N_i^*$, for any $t, s > 0$, $pt N_i^{-1}(s) \leq N_i^*(pt) + s$. If $\inf\Big\{t > 0: \tsum_{i \leq p} N_i^*(p a_i / t) \leq p \Big\} \leq 1$, by the contraction principle (Lemma 4.6, \cite{ledoux1991probability}),
\begin{equation*}
\begin{aligned}
    \Exp\left[ \abs*{\tsum_{i\leq p} a_i Z_i}^p \right]^{1/p}
    & \leq \Exp\left[ \abs*{\frac{1}{p} \tsum_{i\leq p} \epsilon_i(N_i^*(pa_i) + \abs{Y_i^2})}^p \right]^{1/p} \\
    & \leq \frac{1}{p} \left( \tsum_{i \leq p} N_i^*(pa_i)
    + \Exp\left[ \abs*{ \tsum_{i\leq p} \abs{Y_i^2}}^p \right]^{1/p} \right) \\
    & \leq 3.
\end{aligned}
\end{equation*}
By the homogeneity of the $p$-norm, we get $\Exp\left[ \abs*{\tsum_{i\leq p} a_i Z_i}^p \right]^{1/p} \leq 3 \inf\Big\{t > 0: \tsum_{i \leq p} N_i^*(p a_i / t) \leq p \Big\}$. This proves the right-hand side.

For the left-hand side, let $b_1, \dots, b_n \geq 0$ be such that $\tsum_{i=1}^n N_i(b_i) = p$ and
\begin{equation*}
    \tsum_{i \leq p} a_i b_i = \sup\left\{\tsum_{i \leq p} a_ib_i: \tsum_{i \leq p} N_i(b_i) \leq p\right\}.
\end{equation*}
Then,
\begin{equation*}
\begin{aligned}
    \Exp\left[\abs*{\tsum_{i=1}^n a_i Z_i}^p\right]^{1/p}
    & \geq \Exp\left[\abs*{\tsum_{i \leq p} a_i Z_i}^p\right]^{1/p} \\
    & \geq \tsum_{i \leq p} a_i b_i \cdot {\textstyle\prod}_{i \leq p} \Pr[Z_i \geq b_i]^{1/p} \\
    & = \tsum_{i \leq p} a_i b_i \cdot 2^{-\floor{p}/p} \exp\left(-\frac{1}{p}\tsum_{i \leq s} N_i(b_i)\right) \\
    & = \frac{2^{-\floor{p}/p}}{e} \sup\left\{\tsum_{i \leq p} a_ib_i: \tsum_{i \leq p} N_i(b_i) \leq p\right\}. 
\end{aligned}
\end{equation*}
Because the duality between $N_i$ and $N_i^*$ implies 
\begin{equation*}
    \sup\left\{\tsum_{i \leq p} a_ib_i: \tsum_{i \leq p} N_i(b_i) \leq p\right\}
    \geq \inf\Big\{t > 0: \tsum_{i \leq p} N_i^*(p a_i / t) \leq p \Big\},
\end{equation*}
the moment $\tsum_{i=1}^n a_i Z_i$ satisfies
\begin{equation*}
\begin{aligned}
    \Exp\left[\abs*{\tsum_{i=1}^n a_i Z_i}^p\right]^{1/p}
    \geq \frac{1}{2e} \inf\Big\{t > 0: \tsum_{i \leq p} N_i^*(p a_i / t) \leq p \Big\}. 
\end{aligned}
\end{equation*}
By Jensen's inequality,
\begin{equation*}
    \Exp\left[\abs*{\tsum_{i=1}^n a_i Z_i}^p\right]^{1/p}
    \geq \Exp\left[\abs*{\tsum_{i=1}^n a_i \epsilon_i \Exp[\abs{Z_i}]}^p\right]^{1/p}
    \geq \min\{\Exp[\abs{Z_i}]: i \in [n]\} \sqrt{p} \left(\tsum_{i > p} a_i^2 \right)^{1/2}.
\end{equation*}

\end{proof}

\begin{proof}[Proof of Lemma~\ref{thm:log_phi_bound}]

Because $\varphi_p$ is symmetric, positive in $\reals$, and monotonic in $[0,\infty)$,
\begin{equation*}\begin{aligned}
    \phi_p\left(\frac{Z_i}{\eta}\right) 
    & = \int_0^\infty \Pr\left[\varphi_p\left(\frac{Z_i}{\eta}\right) \geq x\right] dx \\
    & = 1 + \int_0^\infty \frac{1}{\eta} \varphi_p'\left(\frac{t}{\eta}\right) \Pr[\abs{Z_i} \geq t] dt \\
    & = 1 + \int_0^\infty \frac{1}{\eta} \varphi_p'\left(\frac{t}{\eta}\right) \max\{\Pr[\abs{Y_i} \geq t], \Pr[L_i \abs{Y_i}^{2/\alpha} \geq t]\} dt.
\end{aligned}\end{equation*}

For the lowerbound,
\begin{equation} \label{eq:log_phi_X_lower_by_log_phi_Y}
\begin{aligned}
    \phi_p\left(\frac{Z_i}{\eta}\right) 
    & = 1 + \int_0^\infty \frac{1}{\eta} \varphi_p'\left(\frac{t}{\eta}\right) \max\{\Pr[\abs{Y_i} \geq t], \Pr[L_i \abs{Y_i}^{2/\alpha} \geq t]\} dt \\
    & \geq 1 + \max\left\{ 
        \int_0^\infty \frac{1}{\eta} \varphi_p'\left(\frac{t}{\eta}\right) \Pr[\abs{Y_i} \geq t] dt, 
        \int_0^\infty \frac{1}{\eta} \varphi_p'\left(\frac{t}{\eta}\right) \Pr[L_i \abs{Y_i}^{2/\alpha} \geq t] dt 
    \right\} \\
    & = \max\left\{ \phi_p\left(\frac{Y_i}{\eta}\right), \phi_p\left(\frac{L_i Y_i^{2/\alpha}}{\eta}\right) \right\}.
\end{aligned}\end{equation}

For the upperbound,
\begin{equation} \label{eq:log_phi_X_upper_by_log_phi_Y}
\begin{aligned}
    \phi_p\left(\frac{Z_i}{\eta}\right) 
    & = 1 + \int_0^\infty \frac{1}{\eta} \varphi_p'\left(\frac{t}{\eta}\right) \max\{\Pr[\abs{Y_i} \geq t], \Pr[L_i \abs{Y_i}^{2/\alpha} \geq t]\} dt \\
    & \leq 1 + \int_0^\infty \frac{1}{\eta} \varphi_p'\left(\frac{t}{\eta}\right) (\Pr[\abs{Y_i} \geq t] + \Pr[L_i \abs{Y_i}^{2/\alpha} \geq t]) dt \\
    & \leq 1 + 2 \max\left\{ 
        \int_0^\infty \frac{1}{\eta} \varphi_p'\left(\frac{t}{\eta}\right) \Pr[\abs{Y_i} \geq t] dt, 
        \int_0^\infty \frac{1}{\eta} \varphi_p'\left(\frac{t}{\eta}\right) \Pr[L_i \abs{Y_i}^{2/\alpha} \geq t] dt 
    \right\} \\
    & = 2 \max\left\{ \phi_p\left(\frac{Y_i}{\eta}\right), \phi_p\left(\frac{L_i Y_i^{2/\alpha}}{\eta}\right) \right\} - 1 \\
    & \leq \exp\left( 2 \max\left\{ \log\phi_p\left(\frac{Y_i}{\eta}\right), \log\phi_p\left(\frac{L_iY_i^{2/\alpha}}{\eta}\right) \right\} \right).
\end{aligned}\end{equation}

Because $Y_i$ is a symmetric random variable with a logarithmically convex tail, we follow the example in Section 3.2 of \cite{latala1997estimation} to obtain
\begin{equation} \label{eq:log_phi_Y}
    p \min\left\{1, \frac{p}{\eta^2 e^6} \right\}
    \leq \log \phi_p \left(\frac{Y_i}{\eta}\right) 
    \leq \frac{16p^2}{\eta^2}.
\end{equation}
On the other hand, because $Y_i^{2/\alpha}$ has a logarithmically concave tail, the example in Section 3.3 of \cite{latala1997estimation} gives 
\begin{equation} \label{eq:log_phi_Y_2/alpha}
\begin{aligned}
    & p \min\left\{1, \max\left\{
        \frac{L_i^p}{e^{2p}\eta^p} \norm{Y_i^{2/\alpha}}_p^p, 
        \frac{p L_i^2}{e^4\eta^2} \norm{Y_i^{2/\alpha}}_2^2\right\} \right\} \\
    & \leq \log \phi_p \left(\frac{L_i Y_i^{2/\alpha}}{\eta}\right) 
    \leq p \max\left\{
        \frac{e^{2p} L_i^p}{\eta^p} \norm{Y_i^{2/\alpha}}_p^p, 
        \frac{e^4 p L_i^2}{\eta^2} \norm{Y_i^{2/\alpha}}_2^2\right\}.
\end{aligned}\end{equation}

The moments of $Y_i$ are given by the integration by parts, 
\begin{equation*}
\begin{aligned}
    \norm{Y_i}_p^p 
    = \int_0^\infty p t^{p-1} \exp(-t^2) dt
    = \frac{p}{2} \int_0^\infty s^{\frac{p}{2}-1} \exp(-s) ds
    = \Gamma\left(\frac{p}{2} + 1\right),
\end{aligned}
\end{equation*}
where Stirling's approximation derives
\begin{equation} \label{eq:moment_Y}
    \norm{Y_i}_p^p = \Gamma\left(\frac{p}{2}+1\right) = \sqrt{\frac{2\pi}{p}} \left(\frac{p}{2e}\right)^{\frac{p}{2}} \left(1 + O\left(\frac{2}{p}\right)\right).
\end{equation}

Plugging \cref{eq:log_phi_Y,eq:log_phi_Y_2/alpha,eq:moment_Y} to \cref{eq:log_phi_X_lower_by_log_phi_Y,eq:log_phi_X_upper_by_log_phi_Y}, we obtain
\begin{equation*}\begin{aligned}
    & p \min\left\{ 1, \max\left\{ \frac{p}{\eta^2 e^6}, 
    \frac{p L_i^2}{e^4 \eta^2} \sqrt{\frac{\alpha\pi}{2}} \left(\frac{2}{\alpha e}\right)^{2/\alpha},
    \frac{L_i^p}{e^{2p} \eta^p} \sqrt{\frac{\alpha \pi}{p}} \left(\frac{p}{\alpha e}\right)^{p/\alpha}
    \right\}\right\}\\
    & \leq \log \phi_p\left( \frac{Z_i}{\eta} \right) 
    \leq 2 p \max\left\{ \frac{16 p}{\eta^2}, 
    \frac{e^4 p L_i^2}{\eta^2} \sqrt{\frac{\alpha\pi}{2}} \left(\frac{2}{\alpha e}\right)^{2/\alpha},
    \frac{e^{2p} L_i^p}{\eta^p} \sqrt{\frac{\alpha \pi}{p}} \left(\frac{p}{\alpha e}\right)^{p/\alpha}
    \right\}.
\end{aligned}\end{equation*}

\end{proof}

\section{GBO norm of \texorpdfstring{$Z_i$}{Zi}} \label{sec:GBO_Z_i}

Suppose that a symmetric random variable $Z$'s survival function is
\begin{equation*}
    \Pr[\abs{Z} \geq t] = \exp(-\min\{ t^2, (t/L)^\alpha \}).
\end{equation*}
Then, for $\eta > 0$,
\begin{equation*}
\begin{aligned}
    \Exp\left[\phi_{\alpha,L}\left( \frac{\abs{Z}}{\eta} \right)\right] 
    & = \int_0^\infty \Pr\left[\phi_{\alpha,L}\left(\frac{\abs{Z}}{\eta}\right) \geq t \right] dt \\
    & = \int_1^\infty \Pr\left[\exp\left(
        \min\left\{\frac{\abs{Z}^2}{\eta^2}, \left(\frac{\abs{Z}}{\eta L}\right)^\alpha \right\} 
    \right) \geq t \right] dt \\
    & = \int_0^\infty \Pr\left[ \min\left\{ \frac{\abs{Z}^2}{\eta^2}, \left(\frac{\abs{Z}}{\eta L}\right)^\alpha \right\} \geq s \right] e^s ds \\
    & = \int_0^{L^{2\alpha/(\alpha-2)}} \Pr\left[ \frac{\abs{Z}^2}{\eta^2} \geq s \right] e^s ds
    + \int_{L^{2\alpha/(\alpha-2)}}^\infty \Pr\left[ \left(\frac{\abs{Z}}{\eta L}\right)^\alpha \geq s \right] e^s ds \\
    & = \int_0^{L^{2\alpha/(\alpha-2)}} \Pr\left[ \abs{Z} \geq \eta \sqrt{s} \right] e^s ds
    + \int_{L^{2\alpha/(\alpha-2)}}^\infty \Pr\left[\abs{Z} \geq \eta L s^{1/\alpha} \right] e^s ds. \\
\end{aligned}
\end{equation*}
Because $\Pr[\abs{Z} \geq t] = \max\{\exp(-t^2), \exp(-(t/L)^\alpha)\}$,
\begin{equation*}
\begin{aligned}
    \Exp\left[\phi_{\alpha,L}\left( \frac{\abs{Z}}{\eta} \right)\right] 
    & = \int_0^{L^{2\alpha/(\alpha-2)}} \Pr\left[ \abs{Z} \geq \eta \sqrt{s} \right] e^s ds
    + \int_{L^{2\alpha/(\alpha-2)}}^\infty \Pr\left[\abs{Z} \geq \eta L s^{1/\alpha} \right] e^s ds \\
    & \geq \int_0^{L^{2\alpha/(\alpha-2)}} \exp(-\eta^2 s) e^s ds
    + \int_{L^{2\alpha/(\alpha-2)}}^\infty \exp(-\eta^\alpha s) e^s ds \\
    & \geq \int_0^{\infty} \exp(-\max\{ \eta^2, \eta^\alpha \} s) e^s ds \\
    & \geq \begin{cases}
        \infty, & 0 < \eta \leq 1, \\
        \frac{1}{\max\{ \eta^2, \eta^\alpha \} -1}, & 1 < \eta.
    \end{cases}
\end{aligned}
\end{equation*}
Hence,
\begin{equation*}
    \norm{Z}_{\phi_{\alpha, L}} 
    = \inf\{ \eta > 0: \Exp[\phi_{\alpha,L}(\abs{Z}/\eta)] \leq 1 \}
    \geq \min\{\sqrt{2}, 2^{1/\alpha}\}.
\end{equation*}
On the other hand, based on \cref{thm:Z_by_Y},
\begin{equation*}
\begin{aligned}
    \Exp\left[\phi_{\alpha,L}\left( \frac{\abs{Z}}{\eta} \right)\right] 
    & \leq \Exp\left[\phi_{\alpha,L}\left( \frac{\max\{\abs{Y}, L \abs{Y}^{2/\alpha}\}}{\eta} \right)\right] \\
    & \leq \Exp\left[\phi_{\alpha,L}\left( \frac{\abs{Y}}{\eta} \right)\right] 
    + \Exp\left[\phi_{\alpha,L}\left( \frac{ L \abs{Y}^{2/\alpha} }{\eta} \right)\right],
\end{aligned}
\end{equation*}
where $Y$ is a symmetric sub-Gaussian random variable satisfying $\Pr[\abs{Y} \geq t] = e^{-t^2}$. Because $\phi_{\alpha,L}(t) = \exp(\min\{t^2, (t/L)^\alpha\})-1 = \min\{ \exp(t^2) - 1, \exp((t/L)^\alpha) - 1\} =  \min\{\psi_2(t), \psi_\alpha(t/L)\}$,
\begin{equation*}
\begin{aligned}
    \Exp\left[\phi_{\alpha,L}\left( \frac{\abs{Z}}{\eta} \right)\right] 
    & \leq \Exp\left[\phi_{\alpha,L}\left( \frac{\abs{Y}}{\eta} \right)\right] 
    + \Exp\left[\phi_{\alpha,L}\left( \frac{ L \abs{Y}^{2/\alpha} }{\eta} \right)\right] \\
    & \leq \Exp\left[\psi_{2}\left( \frac{\abs{Y}}{\eta} \right)\right] 
    + \Exp\left[\psi_{\alpha}\left( \frac{ \abs{Y}^{2/\alpha} }{\eta} \right)\right].
\end{aligned}
\end{equation*}
For any $\alpha > 0$,
\begin{equation*}
\begin{aligned}
    \Exp[\psi_\alpha(Y^{2/\alpha}/\eta)]
    & = \Exp[\exp(Y_i^2/\eta^{\alpha}) - 1] \\
    & = \int_0^\infty \Pr[\exp(Y^2/\eta^\alpha) - 1 \geq s] ds \\
    & = \int_0^\infty \Pr[Y \geq \eta^{\alpha/2}\sqrt{\log(1+s)}] ds \\
    & = \int_0^\infty \frac{2t}{\eta^\alpha} \Pr[Y \geq t] \exp(t^2/\eta^\alpha) dt \\
    & = \int_0^\infty \frac{2t}{\eta^\alpha} \exp( - (1 - 1/\eta^\alpha) t^2) dt \\
    & = \frac{1}{\eta^\alpha},
\end{aligned}
\end{equation*}
where $s = \exp(t^2/\eta^\alpha) - 1$. Therefore,
\begin{equation*}
    \Exp\left[\phi_{\alpha,L}\left( \frac{\abs{Z}}{\eta} \right)\right] 
    \leq \frac{1}{\eta^2 - 1} + \frac{1}{\eta^\alpha - 1}
    \leq \frac{2}{\min\{\eta^2, \eta^\alpha\} - 1},
\end{equation*}
and
\begin{equation*}
    \norm{Z}_{\phi_{\alpha, L}} 
    = \inf\{ \eta > 0: \Exp[\phi_{\alpha,L}(\abs{Z}/\eta)] \leq 1 \}
    \leq \max\{\sqrt{3}, 3^{1/\alpha}\}.
\end{equation*}

\end{document}